\documentclass{article}

\usepackage{amsmath, amsfonts, amssymb}  
\usepackage{enumerate}
\usepackage{enumitem}
\usepackage{graphicx}
\usepackage{hyperref}
\usepackage[all]{xy}
\usepackage{wrapfig}
\usepackage{fancyvrb}
\usepackage[mathscr]{euscript}
\usepackage[T1]{fontenc}
\usepackage{listings}
\usepackage{tikz}
\usepackage{dsfont}
\usepackage{tikz-cd}
\usepackage{centernot}
\usepackage{mathtools}
\usepackage[standard]{ntheorem}
\usepackage{cite}

\newcounter{dummy}
\numberwithin{dummy}{section}
\renewtheorem{lemma}[dummy]{Lemma}
\renewtheorem{proposition}[dummy]{Proposition}
\renewtheorem{theorem}[dummy]{Theorem}
\renewtheorem{definition}[dummy]{Definition}
\renewtheorem{corollary}[dummy]{Corollary}

\renewenvironment{remark}{\vspace{\topsep}\noindent\stepcounter{dummy}\textbf{Remark \arabic{section}.\arabic{dummy}}}{ \vspace{\topsep}}
\renewenvironment{example}{\vspace{\topsep}\noindent\stepcounter{dummy}\textbf{Example \arabic{section}.\arabic{dummy}}}{ \vspace{\topsep}}

\newenvironment{midsecproof}[1]{\vspace{\topsep} \noindent \textit{Proof of #1.}}{$\square$ \vspace{\topsep}}
\renewenvironment{proof}{\emph{Proof:} }{$\square$}

\makeatletter
\DeclareRobustCommand{\cev}[1]{%
  {\mathpalette\do@cev{#1}}%
}
\newcommand{\do@cev}[2]{%
  \vbox{\offinterlineskip
    \sbox\z@{$\m@th#1 x$}%
    \ialign{##\cr
      \hidewidth\reflectbox{$\m@th#1\vec{}\mkern4mu$}\hidewidth\cr
      \noalign{\kern-\ht\z@}
      $\m@th#1#2$\cr
    }%
  }%
}
\makeatother

\title{Modular Character Sheaves on Reductive Lie Algebras}
\author{Colton Sandvik}
\date{}

\newcommand{\C}{\mathbb{C}}
\newcommand{\Q}{\mathbb{Q}}
\newcommand{\Z}{\mathbb{Z}}
\newcommand{\F}{\mathbb{F}}
\newcommand{\E}{\mathbb{E}}
\renewcommand{\O}{\mathbb{O}}
\newcommand{\K}{\mathbb{K}}
\newcommand{\A}{\mathbb{A}}
\newcommand{\G}{\mathbb{G}}

\renewcommand{\sl}{\mathfrak{sl}}

\newcommand{\scrF}{\mathcal{F}}
\newcommand{\scrI}{\mathcal{I}}
\newcommand{\scrG}{\mathcal{G}}
\newcommand{\scrH}{\mathcal{H}}
\newcommand{\scrC}{\mathcal{C}}
\newcommand{\scrD}{\mathcal{D}}
\newcommand{\scrE}{\mathcal{E}}

\newcommand{\scrA}{\mathcal{A}}
\newcommand{\scrK}{\mathcal{K}}
\newcommand{\calB}{\mathcal{B}}
\newcommand{\scrB}{\mathcal{B}}
\newcommand{\scrL}{\mathcal{L}}

\DeclareMathOperator{\Perv}{Perv}
\DeclareMathOperator{\Loc}{Loc}
\DeclareMathOperator{\Ad}{Ad}
\DeclareMathOperator{\Adm}{Adm}

\DeclareMathOperator{\Cusp}{Cusp}
\DeclareMathOperator{\Irr}{Irr}
\newcommand{\CS}{\mathfrak{C}}
\newcommand{\Orb}{\mathfrak{O}}

\DeclareMathOperator{\Hom}{Hom}
\DeclareMathOperator{\im}{im}

\DeclareMathOperator{\supp}{supp}

\newcommand{\D}{\mathbb{D}}
\DeclareMathOperator{\IC}{IC}
\DeclareMathOperator{\Av}{Av}

\DeclareMathOperator{\Res}{Res}
\newcommand{\RRes}{\vec{\underline{\textnormal{Res}}}}
\newcommand{\LRes}{\cev{\underline{\textnormal{Res}}}}

\newcommand{\nRRes}{\vec{\textnormal{Res}}}
\newcommand{\nLRes}{\cev{\textnormal{Res}}}
\newcommand{\Ind}{\underline{\textnormal{Ind}}}
\DeclareMathOperator{\nInd}{Ind}
\newcommand{\T}{\mathbb{T}}

\DeclareMathOperator{\pr}{pr}

\DeclareMathOperator{\For}{For}

\DeclareMathOperator{\id}{id}

\newcommand{\scrO}{\mathcal{O}}
\newcommand{\scrQ}{\mathcal{Q}}
\newcommand{\scrN}{\mathcal{N}}

\DeclareMathOperator{\SL}{SL}

\newcommand{\fr}[1]{\mathfrak{#1}}

\DeclareMathOperator{\scra}{a}
\DeclareMathOperator{\scrb}{b}

\begin{document}
	\maketitle

	\begin{abstract}
		Mirković \cite{M} introduced the notion of character sheaves on a Lie algebra. Due to their simple geometric characterization, character sheaves on Lie algebras can be thought of as a simplified model for Lusztig's theory of character sheaves on algebraic groups.
		We extend the theory to the case of modular coefficients. Along the way, we will reprove some of Mirković's results and provide connections with the modular generalized Springer correspondence of \cite{AHJR1}.
	\end{abstract}

	\tableofcontents

	\section{Introduction}

	Character sheaves were introduced by Lusztig in the 80s and provide a geometric theory of characters \cite{CS1},\cite{CS2}, \cite{CS4}, \cite{CS5}. 
	Character sheaves are certain $G$-equivariant perverse sheaves on $G$ and they have played an important role in the character theory for finite reductive groups.
	In 1986, Lusztig developed an analogous theory of admissible complexes on reductive Lie algebras. These could be seen as a precursor to the theory of character sheaves on reductive Lie algebras.
	Lusztig's original definition, while well motivated, was seemingly complicated compared to how simple characters are in classical representation theory.
	In 1988, Mirković and Vilonen provided a remarkably simple characterization of character sheaves in terms of singular support \cite{MV}. Independently, Ginzburg in 1989 developed a similar notion using symmetric spaces \cite{G}.
	Motivated by these singular support ideas, Mirković fleshed out the theory of character sheaves on reductive Lie algebras \cite{M}. 
	When $\fr{g}$ is over $\C$, then the exponential map intertwines parabolic induction and restriction functors on $\fr{g}$ with $G$ allowing one to give an identification of sorts between the two notions of character sheaves (6.6, \cite{M}).
	While these ideas do not carry over in the modular setting, this hopefully motivates some of the usefulness in studying character sheaves on reductive Lie algebras.

	Character sheaves, along with many classical results in geometric representation theory, provide a connection between geometry (via tools like perverse sheaves and $D$-modules) and characteristic 0 representation theory.
	Over the last several decades, there has been a broad program to provide analogs to many of these classical results in characteristic 0 geometric representation theory (called the ordinary case) to the positive characteristic situations (called the modular case).
	The purpose of this paper is to provide a modular version of the results in \cite{M}. As such we develop a preliminary notion of modular character sheaves. 
	An eventual goal is to motivate future development towards modular character sheaves on reductive groups. 
	The Lie algebra case is significantly easier due to tools not available in the reductive group case, namely the Fourier transforms of Deligne (in the étale setting) and Sato (in the analytic setting).

	Another source of motivation for the theory of modular character sheaves on reductive Lie algebras comes from the modular generalized Springer correspondence of Achar, Henderson, Juteau, and Riche \cite{AHJR1}, \cite{AHJR2}, \cite{AHJR3}.
	We will show that the cuspidal sheaves on $\scrN$ are character sheaves on $\fr{g}$. In addition, we will obtain a decomposition into induction series that strongly resembles that of \cite{AHJR3}.  

	\subsection{Notation}\label{notation}

	Fix primes $\ell \neq p$. 
	Let $\K$ be a finite extension of $\Q_\ell$, the field of $\ell$-adic numbers. 
	Let $\O$ denote the valuation ring of $\K$. Let $\fr{m}$ be the maximal ideal of $\O$. Let $\F = \O/\fr{m}$ be its residue field which is a finite field of characteristic $\ell$. The triple ($\K, \O, \F)$ is called an \textbf{$\fr{l}$-modular system}.
	Throughout, unless otherwise stated, $\E$ will refer to one of $\K$, $\O$, or $\F$. 

	We will fix $G$ to be a connected reductive group over an algebraically closed field $k$ with $\text{char} (k) = p \neq \ell$. Let $\fr{g}$ denote the Lie algebra of $G$ which is a vector space over $k$. 
	The action of $G$ on $\fr{g}$ will always be by the adjoint action. For $g\in G$ and $x \in \fr{g}$, we will write $^g x = \Ad_g (x)$.
	
	We require that $p$ is either very good for $G$ or that $G$ is $A_n$. As a result, there exists a $G$-equivariant non-degenerate symmetric bilinear form on $\fr{g}$ (see [\cite{Let05}, 2.5] for potential improvements of this characteristic requirement). 
	As a result, we can fix a $G$-equivariant isomorphism of $\fr{g}$ with its dual $\fr{g}'$.

	We will work with the following sheaf-theoretic settings, the first of which we will call the étale setting and the second of which we will call the analytic setting.

	\begin{enumerate}
		\item For a variety $X$ defined over $k=\overline{\F_p}$, let $D (X, \E)$ denote the bounded derived category of étale constructible sheaves on $X$ with coefficients in $\E$. 
		If $X$ is a $G$-variety, let $D_G (X, \E)$ denote the corresponding $G$-equivariant derived category of sheaves on $X$ with $\E$-coefficients as in \cite{BL06}.
		\item For a variety $X$ defined over $k=\C$ with a $\C^\times$-action, let $D' (X, \E)$ denote the bounded derived category of constructible sheaves on $X$ with analytic topology and coefficients in $\E$. Let $D(X,\E)$ denote the full subcategory of $D' (X, \E)$ consisting of conic objects, that is sheaves whose restrictions to $\C^\times$-orbits are locally constant.
		In the equivariant setting and when $X$ is a $G$-bundle, we will often abuse notation and write $D_G (X, \E)$ to denote the $\G_m \times G$-equivariant derived category where the $\G_m$-action comes from scaling fibers of $X$.
	\end{enumerate}

	In either setting, we will write $\Perv (\fr{g}, \E)$ for the heart of the perverse $t$-structure in $D (\fr{g}, \E)$. 
	If $f : X \to Y$ is a morphism of algebraic varieties and $\scrA$ is a sheaf on $Y$, then we will denote $f^\circ \scrA = f^! \scrA [\dim Y - \dim X]$.

	\subsection{Acknowledgements}

	I am thankful to Pramod Achar, Joseph Dorta, Arnaud Eteve, and Kostas Psarmoligkos for helpful comments and discussions that influenced this project.

	This work is part of my PhD thesis, advised by Pramod Achar, to whom I am immensely grateful for his mentorship, patience, and guidance throughout.
	
	The author was partially supported by NSF Grant DMS-2231492.
	
	\section{Radon Transforms}

	In this section, we will develop the Radon transforms which give rise to the definitions of character sheaves.
	We will also detail a version of monodromic sheaves on the Lie algebra and use these to define character sheaves on an abelian Lie algebra. Before we can discuss the Radon transform, we will review the Fourier-Deligne transform.

	\subsection{Fourier-Deligne Transform}
	For now, we will restrict ourselves to the étale setting. The theory of the Fourier-Deligne transform for non-modular coefficients is thoroughly covered in \cite{L87}. Laumon's results are easily generalized to the modular case. The labor of extending Laumon's results to the modular case are in Juteau's thesis \cite{juteauThesis}. We will simply recall the definition below.

	We require that $\E^\times$ contains a primitive root of unity of order $p$. We fix a nontrivial character $\psi : \F_p \to \E^\times$. 
	Let $\scrL_\psi$ be the locally constant $\E$-sheaf of rank $1$ on $\G_a$ associated to $\psi$, called the Artin-Schreier local system. It has a rigid structure at zero:
	\[ (\scrL_\psi)_0 \cong \E\]
	If $+ : \A^1 \times \A^1 \to \A^1$ is the addition map, then we have an isomorphism
	\[ \scrL_\psi \boxtimes \scrL_\psi \cong +^* \scrL_\psi  \]
	If $f,g : X \to \A^1$ are morphisms, and $(f,g) : X \to \A^1 \times \A^1$ the map with $f$ and $g$ as the components, we can define $f + g = + (f,g)$. We have the following chain of isomorphisms
	\begin{align*}
		(f+g)^* \scrL_\psi &\cong (f,g)^* +^* \scrL_\psi \\
		&\cong (f,g)^* (\scrL_\psi \boxtimes \scrL_\psi)  \\
		&\cong f^* \scrL_\psi \otimes^L g^* \scrL_\psi 
	\end{align*}

	Let $S$ be a variety and $\pi : E \to S$ be a vector bundle of constant rank $r \geq 1$. We have a dual bundle $\pi' : E' \to S$ and a canonical pairing $\mu : E \times_S E' \to \A^1$.
	Let $\pr : E \times_S E' \to E$ and $\pr' : E \times_S E' \to E'$ denote the obvious projections. We have the following diagram which will encode the Fourier-Deligne transform:

	{ \centering
	\begin{tikzcd}
		& E \times_S E' \arrow[ld, "\pr"'] \arrow[rd, "\pr'"] \arrow[rr, "\mu"] &                       & \A^1 \\
	E \arrow[rd, "\pi"'] &                                                                       & E' \arrow[ld, "\pi'"] &      \\
		& S                                                                     &                       &     
	\end{tikzcd}
	\par }

	\begin{definition}
		The Fourier-Deligne transform for $\pi : E \to S$, associated to the character $\psi$ is the $t$-exact functor (with respect to the perverse $t$-structure)
		$$\T_\psi : D (E, \E) \to D (E', \E)$$
		defined by
		$$\T_\psi (K) = \pr_!' (\pr^* K \otimes^L \mu^* \scrL_\psi)[r].$$
	\end{definition}
	The choice of character is mostly a formality, so we will frequently drop it when the choice is implicit or does not matter.
	Additionally, we will sometimes denote the Fourier-Deligne transform by $\T_E$ to emphasize the vector bundle to which it is being applied.

	We will also briefly recall the notion of convolution for sheaves on a vector spaces. If $\scrA$ and $\scrB$ are sheaves on $V$, then we define the convolution functor 
	\[ D (V, \E) \times D (V, \E) \to D (V, \E) \]
	\[ \scrA \star \scrB = +_* (\scrA \boxtimes \scrB)\]
	for $\scrA, \scrB \in D (V, \E)$.

	There are also equivariant versions of the Fourier-Deligne transform and of convolution. The standard results on the theory of the Fourier-Deligne transform are not changed in the equivariant setting \cite{juteauThesis}.

	To handle the analytic setting, the Fourier-Deligne transform should be replaced with the Fourier-Sato transform. The properties of the Fourier-Sato transform are covered in \cite{KS}, \cite{Br86}. 
	We will primarily work with the Fourier-Deligne transform, and obvious changes can be made to arguments to translate to the analytic setting.

	\subsection{Radon Transforms}

	Let $\calB$ be the flag variety of $G$ of Borel subgroups of $G$. We will primarily identify a Borel subgroup in $\calB$ with its Lie algebra. Fix a Borel subgroup $B$.
	Define the Grothendieck-Springer resolution $\tilde{\fr{g}} = G \times^B \fr{b}$. The space $\tilde{\fr{g}}$ is a $G$-homogenous vector bundle over $\calB$ with fiber $\fr{b}$ at $\fr{b} \in \calB$. 
	Define the Springer resolution $\tilde{\scrN} = G \times^B \fr{n}$ which is a vector bundle over $\calB$ with fibers $[\fr{b}, \fr{b}]$ at $\fr{b} \in \calB$. These spaces come equipped with projections $g : \tilde{\fr{g}} \to \fr{g}$ and $s : \tilde{\scrN} \to \scrN$. 
	The latter of which is the restriction of the Grothendieck-Springer resolution onto the nilpotent cone.

	We can also define $\widetilde{\fr{g}/\fr{n}} = G \times^B \fr{g}/\fr{n}$ which is a vector bundle over $\calB$ with fiber $\fr{g}/[\fr{b}, \fr{b}]$ at $\fr{b} \in \calB$.

	Associated with these resolutions are two Radon transforms. Consider the diagram:

	{ \centering
	\begin{tikzcd}
		& \calB \times \fr{g} \arrow[ld, "p"'] \arrow[rd, "q"] &                       &  &   & {(\fr{b}, x)} \arrow[ld, maps to] \arrow[rd, maps to] &                               \\
 \fr{g} &                                                      & \widetilde{\fr{g}/\fr{n}} &  & x &                                                       & {(\fr{b}, x+[\fr{b}, \fr{b}])}
 \end{tikzcd}
	\par }

	We then have the Grothendieck transform $D(\fr{g}, \E) \to D (\tilde{\fr{g}}, \E)$ given by $\scrG = g^!$ and the horocycle transform $D (\widetilde{\fr{g}/\fr{n}}, \E) \to D (\fr{g}, \E)$ given by $\scrH = q_* \circ p^\circ$.
	Both Radon transforms admit left adjoints $\check{\scrG} = g_! = g_*$ and $\check{\scrH} = p_* \circ q^\circ$.

	\begin{proposition}\label{radonandfouriercompats}
		We have the following compatibilities of Radon transforms.
		\begin{enumerate}
			\item $\T_{\tilde{\fr{g}}} \circ \scrG = \scrH \circ \T_{\fr{g}}$,
			\item $\T_{\widetilde{\fr{g}/\fr{n}}} \circ \scrH = \scrG \circ \T_{\fr{g}'}$,
			\item $\T_{\fr{g}} \circ \check{\scrG} = \check{\scrH} \circ \T_{\tilde{\fr{g}}}$,
			\item $\T_{\fr{g}} \circ \check{\scrH} = \check{\scrG} \circ \T_{\widetilde{\fr{g}/\fr{n}}}$,
			\item $\T_\fr{g} (g_* \underline{\E}_{\tilde{\fr{g}}}) = s_* \underline{\E}_{\tilde{\scrN}} $,
			\item $\check{\scrG} \circ \scrG = g_* \underline{\E}_{\tilde{\fr{g}}} \otimes^L (-)$,
			\item $\check{\scrH} \circ \scrH = s_* \underline{\E}_{\tilde{\scrN}}  \star (-).$
		\end{enumerate}
	\end{proposition}
	\begin{proof}
		\textbf{(1): } Under the $G$-equivariant bilinear form from \ref{notation}, we will identify $\fr{g}' = \fr{g}$. This induces an isomorphism $(\tilde{\fr{g}})' \cong \widetilde{\fr{g}/\fr{n}}$. Note that under this isomorphism, the adjoint of $q : \calB \times \fr{g} \to \widetilde{\fr{g}/\fr{n}}$ becomes the inclusion $j : \tilde{\fr{g}} \to \calB \times \fr{g}$.
		Moreover, it is clear from definitions that $p \circ j = g$. Using the properties of the Fourier-Deligne transform, we obtain
		\begin{align}
			\scrH \circ \T_\fr{g} &= q_* \circ p^\circ \circ \T_\fr{g} \tag{by definition} \\
			&= q_* \circ \T_{\calB \times \fr{g}} \circ p^\circ \tag{base change} \\
			&= \T_{\tilde{\fr{g}}} j^\circ p^\circ \tag{compatibility of morphisms}\\
			&= \T_{\tilde{\fr{g}}} g^\circ \notag \\
			&= \T_{\tilde{\fr{g}}} \scrG \notag
		\end{align}

		\textbf{(2): }
		This is an easy consequence of (1). 
		\begin{align*}
			\T_{\tilde{\fr{g}}} \scrG &= \scrH \circ \T_\fr{g} \\
			\T_{\widetilde{\fr{g}/\fr{n}}} \T_{\tilde{\fr{g}}} \scrG &= \T_{\widetilde{\fr{g}/\fr{n}}} \scrH \circ \T_\fr{g} \\
			\scrG &= \T_{\widetilde{\fr{g}/\fr{n}}} \scrH \circ \T_\fr{g} \\
			\scrG  \T_{\fr{g}'} &= \T_{\widetilde{\fr{g}/\fr{n}}} \scrH \circ \T_\fr{g} \T_{\fr{g}'} \\
			\scrG \T_{\fr{g}'} &= \T_{\widetilde{\fr{g}/\fr{n}}} \scrH
		\end{align*}

		\textbf{(3): } 
		Consider the inclusion maps $i : \tilde{\fr{g}} \to \calB \times \fr{g}$ and $j : \tilde{\scrN} \to \calB \times \fr{g}$. 
		Under our fixed isomorphism of $\fr{g}'$ with $\fr{g}$, $j$ can be identified with $i^\perp$. As a consequence,
		\[ \T_{\calB \times \fr{g}} (i_* \underline{\E}_{\tilde{\fr{g}}} [\dim \calB \times \fr{g}]) = i_*^\perp \underline{\E}_{\tilde{\scrN}} [\dim \tilde{\scrN}] \]
		Of course, $\calB \times \fr{g}$ is simply the trivial bundle over $\calB$ with fibers $\fr{g}$. Let $\mu : \calB \times \fr{g} \to \fr{g}$ be the projection. 
		Hence by the base change compatibility with the Fourier-Deligne transform,
		\[ \mu_* \T_{\calB \times \fr{g}} (i_* \underline{\E}_{\tilde{\fr{g}}} [\dim \calB \times \fr{g}]) = \T_{\calB \times \fr{g}} (\mu_* i_* \underline{\E}_{\tilde{\fr{g}}} [\dim \calB \times \fr{g}]) \]
		Combining the last two results yields
		\[ \mu_* i_*^\perp \underline{\E}_{\tilde{\scrN}} [\dim \tilde{\scrN}]  = \T_\fr{g} (\mu_* i_* \underline{\E} [\dim \calB \times \fr{g}]) \]
		We can simply rewrite this (and use functoriality) to get the desired result.
		\[ \T_\fr{g} (g_* \underline{\E}_{\tilde{\fr{g}}}) = s_* \underline{\E} \]
		
		\textbf{(4), (5): }
		These follow directly since $\check{\scrG}$ is adjoint to $\scrG$ and likewise $\check{\scrH}$ is adjoint to $\scrH$.

		\textbf{(6): } 
		This mostly follows from the projection formula. Let $\scrA \in D (\fr{g}, \E)$.
		\begin{align}
		\check{\scrG} \scrG \scrA &= g_* g^!\scrA \notag \\
		&= g_* (\underline{\E}_{\tilde{\fr{g}}} \otimes^L g^! \scrA) \notag \\
		&= g_* \underline{\E}_{\tilde{\fr{g}}} \otimes^L \scrA \tag{projection formula} 
		\end{align}

		\textbf{(7): }
		We will combine the results of (2), (4), (5) and (6) to establish the claim. Let $\scrA \in D (\fr{g}' , \E)$. Then
		\begin{align}
			s_* \underline{\E}_{\tilde{\scrN}}  \star \T_{\fr{g}'} \scrA &= \T_{\fr{g}'} g_* \underline{\E}_{\tilde{\fr{g}}} \star \T_{\fr{g}'} \scrA \tag{by (5)} \\
			&= \T_\fr{g} \left( g_* \underline{\E}  \otimes^L \scrA \right) \tag{convolution compatibility} \\
			&= \T_{\fr{g}} \left( \check{\scrG} \scrG \scrA \right) \tag{by (6)} \\
			&= \check{\scrH} \scrH \T_{\fr{g}'} \scrA \tag{by (2), (4)} 
		\end{align}
		The result then follows from the fact that $\T$ is an equivalence of categories.
	\end{proof}

	\begin{remark}
		Fix a Borel subalgebra $\fr{b} \in \scrB$, put $\fr{n} = [\fr{b},\fr{b}]$.
		Under the quotient equivalence, we can identify $D_G (\tilde{\fr{g}}, \E)$ with $D_B (\fr{b}, \E)$. 
		Let $j : \fr{b} \hookrightarrow \fr{g}$ and $\pi : \fr{g} \to \fr{g} / \fr{n}$. Under the above identification, the Radon transforms become
		\begin{align*}
			\check{\scrH} &= \Av_{B!}^G  \pi^\circ \\
			\scrH &= \pi_*  \For_B^G \\
			\check{\scrG} &= \Av_{B!}^G  j_* \\
			\scrG &= j^\circ  \For_B^G
		\end{align*}
		where the relevant equivariant sheaf functors ($\Av_{B!}^G$, $\For_B^G$) are denoted and defined as in \cite{A1}. We briefly review the definitions of these functors in \ref{equiv_functor_review}.

		Since $\fr{g}$ is a reductive Lie algebra, we can write $\fr{g} = [\fr{g}, \fr{g}] + Z(\fr{g})$ and $\fr{b} = \fr{b}_s + Z(\fr{g})$.
		We can write $j = j_s + \id_{Z(\fr{g})}$ where $j_s : \fr{b}_s \hookrightarrow [\fr{g}, \fr{g}]$.
		In this setting,
		\begin{align*}
			\check{\scrG}_{[\fr{g}, \fr{g}]} &= \Av_{B!}^G  j_{s*} \\
			\scrG_{[\fr{g}, \fr{g}]} &= j_s^\circ  \For_B^G \\
		\end{align*}
		and the Radon transform on the abelian part is simply the identity functor. Throughout the rest of the paper, we will use these Radon transforms on the Lie algebras themselves rather than the vector bundles over them.
		This will simplify a lot of the writing in the later sections as well. It will also help to make the connection with the ordinary theory of character sheaves more apparent.
		
		We remark that Proposition \ref{radonandfouriercompats} and Corollary \ref{compositionofradonhasidentity} have obvious variants when working in the $\Ad_G$-equivariant setting.
		\end{remark}

		Our first result using this change of perspective will be a useful lemma that will be instrumental in providing a concrete connection between the definitions of character sheaves and cuspidal sheaves.

		\begin{lemma}\label{tensor_and_radon}
			Let $\scrA \in \Perv_G ([\fr{g}, \fr{g}], \E)$ and $\scrL \in \Perv (Z(\fr{g}), \E)$, then
			\[ \check{\scrG} (\scrA \boxtimes \scrL) \cong \scrA \boxtimes \check{\scrG}_{Z(\fr{g})} (\scrL).\]
			Under the identification $\fr{b}' = \fr{g}/\fr{n}$ we also have that
			\[ \check{\scrH} (\T \scrA \boxtimes \T \scrL) \cong \T \scrA \boxtimes \check{\scrH}_{Z(\fr{g})} (\T \scrL).\]
		\end{lemma}
		\begin{proof}
			We can prove the first statement directly using the observations from before.
			\begin{align*}
				\check{\scrG} (\scrA \boxtimes \scrL) &= \Av_{B!}^G  j_* (\scrA \boxtimes \scrL) \\
				&= \Av_{B!}^G \left(j_{s*} \scrA \boxtimes \scrL \right) \\
				&= \Av_{B!}^G j_{s*} \scrA \boxtimes  \scrL \\
				&= \check{\scrG}_{[\fr{g}, \fr{g}]} (\scrA) \boxtimes \scrL
			\end{align*}
			We can recover the second statement using the Fourier-Deligne transform.
			\begin{align*}
				\T \check{\scrH} (\scrA \boxtimes \scrL) &= \T \left( \check{\scrG}_{[\fr{g}, \fr{g}]} (\scrA) \boxtimes \scrL \right) \\
				\check{\scrG} \T (\scrA \boxtimes \scrL) &= \T \left( \check{\scrG}_{[\fr{g}, \fr{g}]} (\scrA) \boxtimes \scrL \right) \\
				\check{\scrG} (\T \scrA \boxtimes \T \scrL) &= \T \check{\scrG}_{[\fr{g}, \fr{g}]} (\scrA) \boxtimes \T \scrL \\
				\check{\scrG} (\T \scrA \boxtimes \T \scrL) &= \check{\scrH}_{[\fr{g}, \fr{g}]} (\T \scrA) \boxtimes \T \scrL.
			\end{align*}
		\end{proof}

	\section{Sheaves on Abelian Lie Algebras}

	In this section, we will cover the theory of character sheaves on abelian Lie algebras. 
	Character sheaves on abelian Lie algebras serve a useful purpose in the general theory by providing a simple description of the character sheaves on a Cartan subalgebra as well as the center of a reductive Lie algebra. 
	The role of character sheaves on a torus is usually played by multiplicative local systems.	The definition can easily be adapted for vector spaces to obtain the notion of an additive local system.
	In the group setting, multiplicative local systems correspond with characters of a torus. In the Lie algebra setting, additive local systems correspond with $\fr{h}'$. We will also explain the replacement for monodromic sheaves in the Lie algebra setting. 
	
	Let $V \subseteq U$ be finite dimensional vector spaces over $k$. We can view both $V$ and $U$ as abelian Lie algebras.

	\begin{definition}\label{defn:abelian_cs}
		A nonzero local system $\scrL \in \Loc (V, \E)$ is an \textbf{additive local system} if $+^\circ \scrL \cong \scrL \boxtimes \scrL$ for the usual addition map $+ : V \times V \to V$. 
		If $\scrL$ is an additive local system, then we call $\scrL [\dim V]$ a \textbf{character sheaf} on $V$. 
	\end{definition}

	One can also characterize additive local systems via their Fourier transforms. 
	\begin{lemma}\label{lem:fourier_dual_char_of_cs}
		Let $\Delta : V' \hookrightarrow V' \times V'$ be the diagonal embedding. Then $\scrL$ is an additive local system on $V$ if and only if $\Delta_* \T_V \scrL \cong \T_V \scrL \boxtimes \T_V \scrL$.
	\end{lemma}
	\begin{proof}
	Note that $\Delta$ is the transpose of $+$.
	If $\scrL$ is an additive local system on $V$, then 
	\begin{align*}
		\Delta_* \T_V \scrL &\cong \T_{V \times V} (+^\circ \scrL) \\
		&\cong \T_{V \times V} (\scrL \boxtimes \scrL) \\
		&\cong \T_{V} (\scrL) \boxtimes \T_{V} (\scrL)
	\end{align*}
	where the first isomorphism follows from the compatibility of the Fourier transform with morphisms and the third isomorphism follows from the compatibility of the Fourier transform with the external tensor product.	
	
	Suppose $\scrL$ is a local system on $V$ that satisfies $\Delta_* \T_V \scrL \cong \T_V \scrL \boxtimes \T_V \scrL$. Then,
	\begin{align*}
		+^\circ \scrL &\cong +^\circ \T_{V'} \T_V \scrL \\
		&\cong \T_{V' \times V'} (\Delta_* \T_V \scrL) \\
		&\cong \T_{V' \times V'} (\T_V \scrL \boxtimes \T_V \scrL) \\
		&\cong \scrL \boxtimes \scrL
	\end{align*}
	Hence, $\scrL$ is an additive local system.
	\end{proof}

	\begin{lemma}\label{abelian_cs:support}
		If $\scrL = \T_V \scrA$ is an additive local system, then
		\begin{enumerate}
			\item $(\supp \scrA)^2 \subset \Delta (V')$;
			\item $\scrA$ is supported at a point.
		\end{enumerate}
	\end{lemma}
	\begin{proof}
		Let $(v,w) \in V' \times V'$. Now, $(\scrA \boxtimes \scrA)_{(v,w)} \cong \scrA_v \otimes^L \scrA_w$.
		As a result, $\supp (\scrA \boxtimes \scrA) = (\supp (\scrA))^2$. Observe $\supp (\Delta_* \scrA) \subset \Delta (V')$. 
		
		By Lemma \ref{lem:fourier_dual_char_of_cs}, $\scrA \boxtimes \scrA \cong \Delta_* \scrA$. As a result, $\supp (\scrA)^2 \subset \Delta (V')$. 
		It follows that $\supp (\scrA)$ is a point.
	\end{proof}

	\begin{lemma}\label{abelian_cs:skyscaper_ls}
		For $\alpha \in V'$, let $i_\alpha : \{\alpha\} \to V'$ denote the inclusion. then
		\begin{enumerate}
			\item In the étale setting, the additive local systems on $V$ are in 1-1 correspondence with $\scrL_\alpha = \T_{V'} (i_{\alpha *} M)$ where $\alpha \in V'$ and $M$ is an $\E$-module. 
			\item In the analytic setting, the additive local systems on $V$ are in 1-1 correspondence with $\scrL_0 = \T_{V'} (i_{0*} M)$ where $M$ is an $\E$-module.
		\end{enumerate}
	\end{lemma}
	\begin{proof}
		We will first handle the étale setting. By Lemma \ref{abelian_cs:support}, all additive local systems have the desired form. It is easy to check that $\scrL_\alpha$ are additive local systems. 
		We can check this easily using the Fourier-dual condition for additive local systems (Lemma \ref{lem:fourier_dual_char_of_cs}):
		\begin{equation}
			\begin{split}
			i_{\alpha *} M \boxtimes i_{\alpha *} M &\cong (i_\alpha \times i_\alpha)_* M \\
			&= (\Delta \circ i_\alpha)_* M \\
			&= \Delta_* (i_{\alpha *} M).
			\end{split} \notag
		\end{equation}
		
		In the analytic setting, the only (additive) local systems on $V$ are constant sheaves. Similarly, $i_{\alpha *} M$ is conic if and only if $\alpha = 0$. As a result, all additive local systems on $V$ are given as in (2).
	\end{proof}

	Let $q : U' \to V'$ be the quotient map. For any finite subset $\theta \subseteq V'$ we say that a sheaf $\scrB \in \Perv (U, \E)$ is \textbf{$\theta$-monodromic} if the support of $\T \scrB$ lies in $q^{-1} \theta$.
	We will denote $\Perv^\theta (U, \E)$ the full Serre subcategory of $\Perv (U, \E)$ generated by the sheaves with $\theta$-monodromy. 
	We can similarly define $D^\theta (U, \E)$ as the full triangulated subcategory of $D (U, \E)$ such that $\scrA \in D^\theta (U, \E)$ if $^p H^i (\scrA) \in \Perv^\theta (U, \E)$ for all $i \in \Z$. 
	We make the nearly trivial observation that the $\theta$-monodromic categories have simple decompositions into $\{\alpha\}$-monodromic components for $\alpha \in \theta$.
	 \[ \Perv^\theta (U, \E) = \bigoplus_{\alpha \in \theta} \Perv^{\{\alpha\}} (U, \E) \]
	 \[ D^\theta (U, \E) = \bigoplus_{\alpha \in \theta} D^{\{\alpha\}} (U, \E) \]
	In particular, simple monodromic perverse sheaves always are contained in some $\Perv^{\{\alpha\}} (U, \E)$.

	We can give a more categorical description of the $\theta$-monodromic category. Let $D_\theta (U', \E)$ be the full triangulated subcategory of $D (U', \E)$ generated by $q^* \scrA$ where $\scrA \in D(\theta, \E)$.
	One can then observe that
	\[D^\theta (U, \E) = \T_{U'} (D_\theta (U', \E)) \]

	We also define categories of $V'$-\textbf{supported} and \textbf{monodromic} sheaves by considering all such possible $\theta$-monodromies. That is,
	$$D_{s} (U', \E) = \bigcup_{\theta \subset V' } D_{\theta} (U', \E)$$
	and
	$$D^{m} (U, \E)= \bigcup_{\theta\subset V' } D^{\theta} (U, \E).$$

	The above notions of monodromy also admit obvious equivariant variations. In particular, if $U'$ is a $G$-variety, we could define equivariant monodromic $D_{G}^\theta (U, \E)$ and $D_G^m (U, \E)$.

	We will apply these notions of monodromy to the case of reductive Lie algebras. Fix a Borel subalgebra $\fr{b}$ of $\fr{g}$, denote $\fr{n} = [\fr{b}, \fr{b}]$, and $\fr{h} = \fr{b}/\fr{n}$. We have the diagram:

	 { \centering
	 \begin{tikzcd}
		 \fr{h} & \fr{b} \arrow[l, "\nu"', two heads] \arrow[r, "j", hook] & \fr{g} \arrow[r, "\pi"] & \fr{g}/\fr{n} & \fr{h}' \arrow[l, hook']
		 \end{tikzcd}
	 \par }
 
	 Fix a finite subset $\theta \subseteq \fr{h}$. 
	 Let $h : \fr{h}' \to \fr{g}/\fr{n}$ be the inclusion map. Its adjoint is the quotient $\nu : \fr{b} \to \fr{h}$. 
	We will primarily work with $B$-equivariant monodromic sheaves in $D^\theta_B (\fr{g}/\fr{n}, \E)$ with respect to $\nu$.

	\section{Character Sheaves and Orbital Sheaves}
	\subsection{Definition of Character Sheaves}

	\begin{definition}\label{defn:cs_general}
		We define the category of \textbf{perverse character sheaves}, denoted $\Perv_G^{\CS} (\fr{g}, \E)$, to be the full subcategory of $\Perv_G (\fr{g}, \E)$ generated by subquotients of ${}^p H^i \check{\scrH} \scrA$ for $\scrA \in \Perv_B^m (\fr{g}/\fr{n}, \E)$.
		Similarly, one defines the \textbf{derived category of character sheaves}, denoted $D_G^{\CS} (\fr{g}, \E)$ which is generated by sheaves $\scrF \in D_G (\fr{g}, \E)$ such that ${}^p H^i (\scrF) \in \Perv_G^{\CS} (\fr{g}, \E)$.
		We define the \textbf{character sheaves} of $\fr{g}$ to be the simple objects in $\Perv_G^{\CS} (\fr{g}, \E)$.
	\end{definition}

	We claim that the notion of a character sheaf on an abelian Lie algebra \ref{defn:abelian_cs} agrees with Definition \ref{defn:cs_general}. The following lemma makes this statement precise.
	\begin{lemma}\label{equiv_of_cs_defn}
		A perverse sheaf $\scrK$ in $\Perv (Z(\fr{g}), \E)$ is a character sheaf in the sense of Definition \ref{defn:abelian_cs} if and only if it is a character sheaf in the sense of \ref{defn:cs_general}.
	\end{lemma}
	\begin{proof}
		Consider shifted local systems of the form $\scrL_\alpha = \T (i_{\alpha*} \E)$. From Lemma \ref{abelian_cs:skyscaper_ls}, all character sheaves (in the sense of \ref{defn:abelian_cs}) are of this form.
		The only Borel subalgebra of $Z(\fr{g})$ is all of $Z(\fr{g})$. As a result, $\check{\scrH}$ is simply the identity functor.
		The simple objects in $\Perv^m_B (Z (\fr{g}), \E)$ are the Fourier-Deligne transforms of skyscraper sheaves which are precisely the $\scrL_\alpha$'s.
	\end{proof}

	\begin{definition}\label{defn:orbtial}
		Define the \textbf{derived category of orbital sheaves}, denoted $D_G^\Orb (\fr{g}, \E)$, to be the full subcategory of $D_G (\fr{g}, \E)$ whose objects are supported in the closure of a $G$-orbit.
		The heart of this category, denoted $\Perv_G^{\Orb} (\fr{g}, \E)$ is given by $\Perv_G (\fr{g}, \E) \cap D_G^{\Orb} (\fr{g}, \E)$. 
		The simple perverse sheaves in $\Perv_G^{\Orb} (\fr{g}, \E)$ are called \textbf{orbital sheaves} on $\fr{g}$. Orbital sheaves can also be defined as the sheaves whose support is contained in the closure of a single $G$-orbit.
	\end{definition}

	Note that the closure of a $G$-orbit in $\fr{g}$ is a union of finitely many $G$-orbits (\cite{Jan04}, proof of 8.4). As a result, orbital sheaves could alternatively be defined as sheaves supported on finitely many $G$-orbits.

	\begin{proposition}\label{compositionofradonhasidentity}
		Let $\E$ be a field (for example $\E = \K, \F$). Let $\scrA \in \Perv_G^{\Orb} (\fr{g}, \E)$ be a simple orbital sheaf. Then $\scrA$ is a subquotient of $^p H^0 (\check{\scrG} \scrG (\scrA))$.
	\end{proposition}
	\begin{proof}
		Let $\scrA = \IC (\scrO, \scrL)$ for some orbit $\scrO$ and some irreducible local system $\scrL$ on $\scrO$. 
		Consider the cartesian square,

		{ \centering
		\begin{tikzcd}
			\tilde{\fr{g}} \times_{\fr{g}} \scrO \arrow[d, "g'"'] \arrow[r] & \tilde{\fr{g}} \arrow[d, "g"] \\
			\scrO \arrow[r, "j"']                                           & \fr{g}                       
			\end{tikzcd}
		\par }

		We will first show that $\scrL$ is a subobject of $H^0 (g_* g^* \scrL)$. Consider the map given by the adjunction 
		\[ \scrL \to H^0 (g_* g^* \scrL)\]
		This map is nonzero since it comes from the identity map between local systems $g^* \scrL \to g^* \scrL$. Since $\scrL$ is irreducible, we then have that $\scrL$ is a composition factor of $H^0 (g_* g^* \scrL)$.

		By proper base change, $j^* (g_*g^* \scrA) = g'_* g^* \scrL [\dim \fr{g}]$. As a result, some of the composition factors of $^p H^0 (g_* g^* \scrA)$ are given by IC-extensions of local systems on $\scrO$ which are composition factors of $^p H^i (g_* g^* \scrL)$. 
		In particular, $\scrA \cong \IC (\scrO, \scrL)$ is a subquotient of $^p H^0 (g_* g^* \scrA [\dim \fr{g}]) \cong  {}^p H^0 (\check{\scrG} \scrG (\scrA))$.
	\end{proof}

	The major result relating orbital sheaves with character sheaves is stated in Corollary \ref{cs_are_fourier_of_orbital}. To prove this we first need a lemma on geometry of orbits.

	\begin{lemma}\label{hom_alg_reduction_to_simple}
		Let $G,H$ be groups, and let $X$ be a $G$-variety and $Y$ an $H$-variety. 
		Suppose $F : \Perv_G (X, \E) \to \Perv_H (Y, \E)$ is an exact functor. Let $\scrB \in \Perv_G (X, \E)$ and suppose $\scrA \hookrightarrow F(\scrB)$ is a simple subquotient.
		Then there exists a simple $\scrC \in \Perv_G (X, \E)$ such that $\scrA$ is a subquotient of $F (\scrC)$.
	\end{lemma}

	\begin{lemma}[\!\!\textnormal{\cite{M}, Lemma 3.3}]\label{orbit_calc_lemma}
		\begin{enumerate}
			\item  Let $x = s + n \in \fr{l}$ be a Jordan decomposition with $x$ semisimple and $n$ nilpotent. Pick a parabolic subgroup $P$ of $G$ with Levi decomposition $P = L \ltimes U$ ( with Lie algebras $\fr{p} = \fr{l} \oplus \fr{n}$) such that $\fr{l} = Z_g (s)$. Then $^U x = x + \fr{n}$ for the unipotent subgroup $U$ corresponding to $\fr{n}$.
			\item For any $\alpha \in \fr{h}$ the semisimple components of all elements of $\nu^{-1} \alpha$ are conjugate.
			\item For any $y \in \fr{g}$, $\nu ({}^G y \cap \fr{b})$ is finite.
		\end{enumerate}
	\end{lemma}
	\begin{proof}
		Statements (2) and (3) are proved in \cite{M}, Lemma 3.3 in the Lie algebra setting. For statement (1), Mirković refers to Proposition 2.7 of \cite{ICC} which is phrased in the algebraic group setting.
		For the sake of completeness, we will include the details in proving (1) following \emph{loc. cit.}.

		Let $\pi : \fr{p} \to \fr{l}$ denote the standard projection. Let $u \in U$.
		Since $\pi$ is $U$-invariant, we have that $\pi (x) = \pi ({}^u x)$. Moreover, since $x \in \fr{l}$, one has that ${}^u x \in \fr{p}$. The previous two remarks give that ${}^u x \in x + \fr{u}$, and hence, ${}^U x \subseteq x + \fr{u}$.
		Observe that $\dim ({}^U x) = \dim U = \dim (x + \fr{u})$. Therefore, ${}^U x$ is dense in $x+\fr{u}$. Additionally, since $U$ is a unipotent group and $\fr{g}$ is an affine variety, one has that ${}^U x$ is a closed subvariety of $\fr{g}$.
		Therefore, ${}^U x = x + \fr{u}$ as desired.
	\end{proof}

	\begin{proposition}\label{orbital_and_radon}
		The functor $^p H^0 \check{\scrG}$ restricts to a functor $\Perv_{B, s} (\fr{g}/\fr{n}, \E) \to \Perv_G^{\Orb} (\fr{g}, \E)$. Moreover, if $\scrC$ is an orbital sheaf, then there exists some $\scrA \in \Perv_{B, s} (\fr{g}/\fr{n}, \E)$ such that $\scrC$ is a subquotient of $^p H^0 \check{\scrG}$.
	\end{proposition}
	\begin{proof}	
		We will first prove that $^p H^0 \check{\scrG}$ takes $\fr{h}$-supported complexes to orbital sheaves. Let $\scrD \in \Perv_{B,s} (\fr{g}/\fr{n}, \E)$ be a simple perverse sheaf.
		As a result, $\supp (\scrD) \subseteq \nu^{-1} \alpha$ for some $\alpha \in \fr{h}$. The Radon transform $\check{\scrG}$ does not change the support much; in fact, at worst, it lies in the $G$-orbit of the original support. As a consequence,
		$\supp (\check{\scrG} \scrD) \subseteq G \cdot (\nu^{-1} \alpha)$.
		As a result, $\supp (\check{\scrG} \scrD)$ is contained in the union of $G$-orbits of $\fr{g}$ which meet $\nu^{-1} \alpha$. Since there are only finitely many nilpotent orbits, it suffices to show that semisimple parts of $\nu^{-1} \alpha$ are contained in finitely many $G$-orbits in $\fr{g}$.
		By Lemma \ref{orbit_calc_lemma}, we have that the semisimple parts of $\nu^{-1} \alpha$ are conjugate, so the previous statement does hold. Therefore, $\check{\scrG} \scrD \in D_G^{\Orb} (\fr{g}, \E)$, and so $^p H^0 \check{\scrG} \scrD \in \Perv_G^{\Orb} (\fr{g}, \E)$.
		By a standard argument in homological algebra, one obtains that $^p H^0 \check{\scrG}$ takes $\fr{h}$-supported perverse sheaves to orbital sheaves.

		It remains to check that if $\scrC$ is an orbital sheaf, then we can find it as a subquotient of a sheaf in the essential image of $^p H^0 \check{\scrG}$.  
		By Proposition \ref{compositionofradonhasidentity}, $\scrC$ is a subquotient of $^p H^0 \check{\scrG} (\scrG (\scrC))$. We also have that $\scrG (\scrC) = j^\circ \scrC \in \Perv_{B,s} (\fr{g}/\fr{n})$ which completes the proof.
	\end{proof}

	\begin{corollary}\label{cs_are_fourier_of_orbital}
		The Fourier transform restricts to an equivalence of categories,
		\[\T_{\fr{g}'} : \Perv_G^{\Orb} (\fr{g}', \E) \to \Perv_G^{\CS} (\fr{g}, \E).\]
		In particular, character sheaves are precisely the Fourier transforms of orbital sheaves.
	\end{corollary}
	\begin{proof}
		First we note that $\Perv^m_B (\fr{g}/\fr{n}, \E) = \T_{\fr{b}} (\Perv_{B,s} (\fr{b}, \E))$ and $\T \circ \check{\scrH} = \check{\scrG} \circ \T$.
		Suppose $\scrC$ is an orbital sheaf. 
		By Proposition \ref{orbital_and_radon}, $\scrC$ is a subquotient of $\check{\scrG} (\scrD)$ for $\scrD \in \Perv_{B,s} (\fr{g}/\fr{n}, \E)$.
		We can find some $\scrB \in \Perv^m_B (\fr{g}/\fr{n}, \E)$ such that $\scrD = \T \scrB$. As a consequence, $\scrC$ is a subquotient of $(\check{\scrG} \circ \T) (\scrB) = (\T \circ \check{\scrH}) (\scrB)$. Therefore, $\T \scrC$ is a subquotient of $\check{\scrH} \scrB$.
		Equivalently, $\T \scrC$ is a character sheaf.

		Now suppose $\scrA$ is a subquotient of $\check{\scrH} \scrB$ for $\scrB \in \Perv_B^m (\fr{g}/\fr{n}, \E)$, i.e., $\scrA$ is a character sheaf.
		It follows that $\T \scrA$ is a subquotient of $(\T \circ \check{\scrH} ) (\scrB) = (\check{\scrG} \circ \T) (\scrB)$.
		By Proposition \ref{orbital_and_radon}, $\T \scrA$ is an orbital sheaf, and hence, $\scrA$ is the Fourier-Deligne transform of an orbital sheaf.
	\end{proof}

	\begin{example}\label{ex:sl2_nuances}
		Note that $g$ is a small resolution. In the characteristic 0 setting, one can then use the decomposition theorem to deduce that $g_* \underline{\E}_{\tilde{\fr{g}}}$ is semisimple.
		This does not hold in the modular setting.

		Consider the case of $\fr{g} = \mathfrak{s}\mathfrak{l}_2 (\C)$ and $\E$ a field of characteristic 2. If $g_* \underline{\E}_{\tilde{\fr{g}}}$ were semisimple, then since $\T$ is an autoequivalence of categories, the Springer sheaf $\scrA = s_* \underline{\E}_{\tilde{\scrN}} [2]$ would be semisimple.
		The adjoint orbits of $ \mathfrak{s}\mathfrak{l}_2 (\C)$ are well studied (see for example, Chapter 3 of \cite{CM}). 
		There are two distinct 2 nilpotent orbits, the orbit $C_0$ of $0$ which is just a point and the regular orbit $C_r$ consisting of $\scrN \setminus \{0\}$. 
		It is well known that
		$$\IC (C_r, \underline{\E})_0 = \begin{cases} \E & i=-2, -1 \\ 0 & \text{else} \end{cases}.$$
		On the other hand,
		$$\scrA \vert_{C_r} = \underline{\E} [2] \:\:\:\:\: \scrA\vert_0 = H^* (\mathbb{P}^1) [2]$$
		So $\scrA$ is not semisimple. Nonetheless, $\IC (C_0,\underline{\E})$ and $\IC (C_r, \underline{\E})$ appear in its composition series.
		More generally, for any reductive group $G$, one may note that $\IC (C_0, \underline{\E})$, the skyscraper sheaf, is always a simple quotient of the Springer sheaf.
	\end{example}

	\subsection{Restriction and Induction Functors}\label{equiv_functor_review}
	
	We will first recall some equivariant sheaf functors. The basic definitions and properties of these functors can be found in \cite{A1}. 
	Let $G$ be an algebraic group and let $X$ be an $G$-variety over $k$. 
	If $H \subset G$ is a closed subgroup, then there is an obvious forgetful functor $\For_H^G : D_G (X, \E) \to D_H (X, \E)$. The forgetful functor admits both left and right adjoints, denoted by 
	\[ \Av_{H!}^G : D_H (X, \E) \to D_G (X,\E) \:\:\:\: \text{and} \:\:\:\: \Av_{H*}^G : D_H (X,\E) \to D_G (X,\E),\]
	respectively. We will briefly describe how these are constructed. Let $\sigma : G \times X \to X$ be the action map of $G$ on $X$. There is an induced map on the induction space $\overline{\sigma} : G \times^H X \to X$.
	Let $i : X \to G \times^H X$ be the obvious map $i(x) = (e,x)$. It is well known that 
	\[i^*[-\dim G/H]  \For_H^G  \cong i^! [\dim G/H] \For_H^G :  D_G (G \times^H X, \E) \to D_H (X, \E) \label{eq:ind_equiv}\] 
	is an equivalence of categories. We then define the averaging functors via
	\[ \Av_{H*}^G = \overline{\sigma}_* \circ (i^* \circ \For_H^G)^{-1}\]
	and
	\[ \Av_{H!}^G = \overline{\sigma}_! \circ (i^! \circ \For_H^G)^{-1}.\]
	
	Let $P$ be a parabolic subgroup of $G$. It has a Levi decomposition $P = L \ltimes U$. Let $\fr{p}, \fr{l}, \fr{u}$ denote their Lie algebras. We have the following Cartesian diagram:

	{ \centering
	\begin{tikzcd}
		& \fr{p} \arrow[ld, "i"'] \arrow[rd, "\pi"] &                        \\
\fr{g} \arrow[rd, "\tau"'] &                                           & \fr{l} \arrow[ld, "j"] \\
		& \fr{g}/\fr{u}                             &                       
\end{tikzcd}
	\par }

	We remark that $j$ is the adjoint of $\pi$ after identifying the $\fr{l}'$ with $\fr{l}$ and noting that $\fr{p}'$ can be identified with $\fr{g}/\fr{u}$ using the fixed $G$-invariant bilinear form. 
	Similarly, $\tau$ is the adjoint of $i$ after identifying $\fr{g}'$ with $\fr{g}$. We define the induction functor $\Ind_P^G : D_L (\fr{l}, \E) \to D_G (\fr{g}, \E)$ by the formula
	\[ \Ind_P^G = \Av_{P!}^G \circ i_! \circ \pi^* \circ \Av_{L!}^P \cong  \Av_{P*}^G \circ i_* \circ \pi^! \circ \Av_{L*}^P \]
	The induction functor admits both left and right adjoints, which are nearly the same with some subtle differences. We will call both adjoints the restriction functor and use notation to indicate which one is being used. The left adjoint will be defined by
	\[ \LRes_P^G = \For_L^P \circ \pi_! \circ i^* \circ \For_P^G \]
	and the right adjoint will be defined by
	\[ \RRes_P^G = \For_L^P \circ \pi_* \circ i^! \circ \For_P^G .\]
	It is a simple consequence of proper base change that we could equivalently define these restriction functors by
	\[ \LRes_P^G = \For_L^P \circ j^* \circ \tau_! \circ \For_P^G \]
	\[ \RRes_P^G = \For_L^P \circ j^! \circ \tau_* \circ \For_P^G. \]
	Moreover, we will sometimes drop $P$ and $G$ from the restriction and induction functors' notation when they are clear from context. We may also include the Levi, such as $\Ind_{L \subset P}^G$, if the choice of Levi is not obvious.

	We will now establish some basic results for these functors.

	\begin{lemma}[\cite{AHJR1}, Lemma 2.14]\label{lem:ind_equiv_version_of_ind}
		There is an isomorphism of functors
		\[\Ind_P^G \cong \mu_! \circ I_P^G \circ \pi^* [2 \dim (G/P)],\]
		where $I_P^G : D_P (\fr{p}, \E) \to D_P (G \times^P \fr{p}, \E)$ is the inverse of the induction equivalence induction (\ref{eq:ind_equiv}) and $\mu : G \times^P \fr{p} \to \fr{g}$ is induced by the adjoint action.
	\end{lemma}

	\begin{lemma}\label{ind_res_verdier_duality_compat}
		We have natural isomorphisms of functors
		\[ \D \Ind_P^G \cong \Ind_P^G \D;\]
		\[ \D \LRes_P^G \cong \RRes_P^G \D.\]
	\end{lemma}
	\begin{proof}
		These can both be computed directly using the fact that Verdier duality commutes with sheaf functors.
		\begin{align*}
			\D \Ind_P^G &= \D \circ \Av_{P!}^G \circ i_! \circ \pi^* \circ \Av_{L!}^P \\
			&= \Av_{P*}^G \circ i_* \circ \pi^! \circ \Av_{L*}^P \circ \D \\
			&= \Ind_P^G \D
		\end{align*}
		By the same reasoning, the compatibility of Verdier duality with parabolic restriction also holds.
	\end{proof}

	\begin{lemma}\label{ind_res_fourier_compat}
		Fourier-Deligne transform commute with parabolic restriction and induction.
	\end{lemma}
	\begin{proof}
		We will prove the statement first for left-restriction and then use the adjoint relation to establish the claim for the others.
		\begin{align}
			\T_\fr{l} \LRes_P^G &= \T_\fr{l} \For_L^P \circ \pi_! \circ i^* \circ \For_P^G \notag \\
			&= \For_L^P \circ \T_\fr{l} \circ \pi_! \circ i^* \circ \For_P^G \notag \\
			&= \For_L^P \circ j^* \circ \T_{\fr{p}} \circ i^* \circ \For_P^G \tag{morphism compatibility} \\
			&=  \For_L^P \circ j^* \circ \tau_! \circ \T_{\fr{g}'} \circ \For_P^G \tag{morphism compatibility} \\
			&= \For_L^P \circ j^* \circ \tau_! \circ \For_P^G \circ \T_{\fr{g}'} \notag \\
			&= \LRes_P^G \T_{\fr{g}'}. \notag
		\end{align}
		We can now prove the statement for induction.
		\begin{align}
			\Hom (\scrB, \Ind \T_{\fr{l}'} \scrA ) &= \Hom (\LRes \scrB, \T_{\fr{l}'} \scrA ) \tag{adjoint relation}\\
			&= \Hom (\T_\fr{l} \LRes \scrB, \scrA) \tag{$\T$ is fully faithful} \\
			&= \Hom (\LRes \T_{\fr{g}'} \scrB, \scrA ) \notag \\
			&= \Hom (\T_\fr{g}' \scrB, \Ind \scrA ) \tag{adjoint relation} \\
			&= \Hom (\scrB, \T_{\fr{g}} \Ind \scrA ) \tag{$\T$ is fully faithful}
		\end{align}
		By Yoneda's lemma, it follows that $\Ind \T_{\fr{l}'} \cong \T_{\fr{g}} \Ind$.
		The argument for right-restriction will be very similar.
		\begin{align*}
			\Hom (\scrB, \RRes \T_{\fr{g}'} \scrA) &= \Hom (\Ind \scrB, \T_{\fr{g}'} \scrA ) \\
			&= \Hom (\T_{\fr{g}} \Ind \scrB, \scrA ) \\
			&= \Hom (\Ind \T_{\fr{l}'} \scrB, \scrA ) \\
			&= \Hom (\T_{\fr{l}'} \scrB, \RRes \scrA ) \\
			&= \Hom (\scrB, \T_{\fr{l}} \RRes \scrA )
		\end{align*}
		By Yoneda's lemma, it follows that $\RRes \T_{\fr{g}'} \cong \T_{\fr{l}} \RRes$ as desired.
	\end{proof}

	Before concluding our discussion of the induction and restriction functors, we will make note of their transitivity.

	\begin{lemma}\label{ind_res_transitivity}
		The restriction and induction functors are transitive. This is to say, we have an isomorphism of functors
		\[ \LRes_{L \subset P}^G \cong \LRes_{L \subset P \cap Q }^M \circ \LRes_{M \subset Q}^{G} \]
		\[ \RRes_{L \subset P}^G \cong \RRes_{L \subset P \cap Q }^M \circ \RRes_{M \subset Q}^{G} \]
		\[ \Ind_{L \subset P}^G \cong \Ind_{M \subset Q }^G \circ \Ind_{L \subset P \cap M}^{M} \]
	\end{lemma}
	\begin{proof}
		This result can be found for the case of the nilpotent cone instead of the Lie algebra in (7.1) of \cite{AHR} for restriction and (7.8) of \cite{AHR} for induction.
		The proof holds for reductive Lie algebra with the obvious modifications.
	\end{proof}

	\section{Cuspidal Sheaves}

	In Lusztig's theory of character sheaves, the cuspidal sheaves act as the building blocks for character sheaves. In [Theorem 23.1, \cite{CS5}], Lusztig establishes that (under mild constraints) cuspidal sheaves are character sheaves. 
	The proof of this is a culmination of the theory built up throughout Lusztig's papers on character sheaves. The main result in this section will be the Lie algebra analog of this statement. However, using the Fourier transform, we can give a simple geometric characterization of cuspidal sheaves. 
	This along with Proposition \ref{cs_are_fourier_of_orbital} will give a straightforward proof that cuspidal sheaves are character sheaves.
	
	\begin{proposition}\label{prop:res_small_conds}
		Let $\scrA$ be a simple perverse sheaf on $\fr{g}$. The following conditions are equivalent:
		\begin{enumerate}
			\item $\RRes_P^G \scrA = 0$ for all $L \subset P \subsetneq G$;
			\item $\LRes_P^G \scrA = 0$ for all $L \subset P \subsetneq G$;
			\item $\scrA$ is not a quotient of $\Ind_P^G (\scrB)$ for any $L \subset P \subsetneq G$ and $\scrB \in \Perv_L (\fr{l}, \E)$;
			\item $\scrA$ is not a subobject of $\Ind_P^G (\scrB)$ for any $L \subset P \subsetneq G$ and $\scrB \in \Perv_L (\fr{l}, \E)$;
		\end{enumerate}
	\end{proposition}
	\begin{proof}
		This is proved for sheaves on the nilpotent cone in Proposition 2.1 of \cite{AHJR1}. Their proof that (1) and (2) are equivalent does not transfer for the Lie algebra in the étale setting.
		Namely, it is not true that $\Perv_G (\fr{g}, \E)$ is generated by weakly $\G_m$-equivariant objects, so hyperbolic localization cannot be used.
		In particular, we cannot deduce a relationship between $\RRes_P^G$ and $\LRes_{P^{-}}^G$ in the étale setting. 
		
		The equivalence of (1) and (2) can still be derived through elementary means. By Lemma \ref{ind_res_verdier_duality_compat}, we have that
		\[\RRes_P^G \scrA \cong \D (\D \RRes_P^G \scrA ) \cong \D (\LRes_P^G \D \scrA)\]
		As a result, $\RRes_P^G \scrA = 0$ if and only if $\D (\LRes_P^G \D \scrA) = 0$. Since $\D$ is an involution of the derived category, we get that $\RRes_P^G \scrA = 0$ if and only if $\LRes_P^G \scrA = 0$ as desired.
	\end{proof}

	\begin{definition}\label{defn:cuspidal}
		Let $\scrA$ be a simple perverse sheaf in $\Perv_G (\fr{g}, \E)$. We call $\scrA$ \textbf{res-small} if it satisfies one of the conditions of Proposition \ref{prop:res_small_conds}.

		We will call $\scrA$ \textbf{$\boxtimes$-factored} if it decomposes
		\[ \scrA = \scrL \boxtimes \scrB \]
		for a sheaf $\scrL$ on $Z(\fr{g})$ and some sheaf $\scrB$ on $[\fr{g}, \fr{g}]$.
		We call the sheaf $\scrL$ the \textbf{central factor} and the sheaf $\scrB$ the \textbf{semisimple factor}.
		If $\scrA$ is both res-small and $\boxtimes$-factored with an additive local system as its central factor, then we say $\scrA$ is \textbf{cuspidal}.
	\end{definition}

	The equivalence of the conditions of Definition \ref{defn:cuspidal} is proved for the nilpotent cone in Proposition 2.1 of \cite{AHJR1}. 
	
	We will write $\Cusp_G (\fr{g}, \E)$ for the Serre subcategory of $\Perv_G (\fr{g}, \E)$ consisting of cuspidal sheaves.

	\begin{theorem}\label{reductive_cusp_are_cs}
		Any cuspidal sheaf on a reductive Lie algebra $\fr{g}$ is a character sheaf.
	\end{theorem}

	The proof of Theorem \ref{reductive_cusp_are_cs} will occupy the remainder of the section. We will briefly outline the strategy: 
	\begin{enumerate}
		\item Show that res-small is stable under the Fourier transform (Lemma \ref{cusp_iff_fourier_is_cuspidal}).
		\item Use the geometry of res-small sheaves to characterize res-small sheaves over a semisimple Lie algebra in terms of their support and the support of their Fourier-Deligne transform (Lemmas \ref{cusp_tech_lemma1}, \ref{supp_of_res_small}, \ref{cusp_and_support_equiv}).
		\item Conclude that for semisimple Lie algebras, cuspidal sheaves are character sheaves (Corollary \ref{semisimple_cusp_are_cs}).
		\item Show that $\boxtimes$-factored sheaves with character sheaf central and semisimple factors are themselves character sheaves (Lemma \ref{equiv_defn_for_cs_factorization}).
		\item Use (4) to reduce the reductive case to the semisimple case, and then use (3) to finish the proof.
	\end{enumerate}

	\begin{lemma}\label{cusp_iff_fourier_is_cuspidal}
		Let $\scrA$ be a simple perverse sheaf on $\fr{g}$. Then $\scrA$ is res-small if and only if $\T_{\fr{g}} \scrA$ is res-small. Additionally, if $\fr{g}$ is semisimple, then $\scrA$ is cuspidal if and only if $\T_{\fr{g}} \scrA$ is cuspidal.
	\end{lemma}
	\begin{proof}
		When $\fr{g}$ is semisimple, since $Z(\fr{g}) = 0$, the $\boxtimes$-factored condition always holds. As a result, $\scrA$ is cuspidal if and only if $\scrA$ is res-small.

		Since $\T$ is its own inverse functor, it suffices to prove this in the direction that if $\scrA$ is res-small, then $\T_{\fr{g}} \scrA$ is res-small.
		Let $P \subsetneq G$ be a proper parabolic subgroup of $G$. By Lemma \ref{ind_res_fourier_compat},
		\[ \RRes_P^G \T_{\fr{g}} \scrA = \T_{\fr{l}'} \RRes_P^G \scrA = \T_{\fr{l}'} (0) = 0.\]
		Thus $\T_{\fr{g}} \scrA$ is res-small. 
	\end{proof}

	We will introduce some notation that will be convenient for the next several lemmas. Let $X$ be a subvariety of a variety $Y$. We will denote $i_{X, Y}$ for the inclusion map $i_{X,Y} : X \hookrightarrow Y$.
	For the special case when $X = \{x\}$ is a point, then we will denote $i_{x,Y}$ instead of $i_{\{x\}, Y}$. 

	\begin{lemma}\label{cusp_tech_lemma1}
		Let $x = s+n$ be the Jordan decomposition of $x \in \fr{l}$. Let $\fr{p} = \fr{l} \oplus \fr{u}$ be a parabolic subalgebra with Levi decomposition.
		Then for any $\scrA \in \Perv_G (\fr{g}, \E)$,
		\[ i_{x,\fr{l}}^! \RRes_P^G \scrA = i_{x, \fr{g}}^! \scrA [2\dim U].\]
	\end{lemma}
	\begin{proof}
		Consider the cartesian square:

		{ \centering
		\begin{tikzcd}
			& x \arrow[ld, "i_{x,\fr{l}}"'] \arrow[dd, "i_{x,\fr{g}/\fr{u}}"] \arrow[r, bend left, "i_{x, x+\fr{u}}"] & x+\fr{u} = {}^Ux \arrow[l, bend left, "a_x"] \arrow[dd, "i"] \\
\fr{l} \arrow[rd, "j"'] &                                              &                                                       \\
			& \fr{g}/\fr{u}                                & \fr{g} \arrow[l, "\tau"]                             
\end{tikzcd}
		\par }

		We can now perform a proper base change to obtain
		\begin{align*}
			i_{x, \fr{l}}^! \RRes_P^G \scrA &= i_{x, \fr{l}}^!  j^! \tau_* \scrA \\
			&= i_{x,\fr{g}/\fr{u}}^!  \tau_* \scrA \\
			&=  a_{x*} i^! \scrA
		\end{align*}

		Since $i^! \scrA$ is $U$-equivariant, it is locally constant on $U$-orbits. We can say something stronger: since $U$ is unipotent (and hence contractible), we have that $i^! \scrA$ has constant cohomology, and so 
		\[ i^! \scrA \cong a_x^! i_{x,x+\fr{u}}^! i^! \scrA \cong a_x^! i_{x,\fr{g}}^! \scrA \].
		We can then apply $a_{x*}$ to finish the proof.
		\begin{align*}
			a_{x*} i^! \scrA &\cong a_{x*} a_x^! i_{x,\fr{g}}^! \scrA \\
			&\cong a_{x*} a_x^* i_{x, \fr{g}}^! \scrA [2\dim U] \\
			&\cong i_{x, \fr{g}}^! \scrA [2\dim U]
		\end{align*}
		where $a_{x*} a_x^*  \cong \text{Id}$ follows from the fact that $\fr{u}$ is contractible.
	\end{proof}

	\begin{lemma}\label{supp_of_res_small}
		Let $\scrA$ be a res-small perverse sheaf in $\Perv_G (\fr{g}, \E)$. Then $\scrA$ is supported on $Z(\fr{g}) + \scrN$.
	\end{lemma}
	\begin{proof}
		There is a smooth $G$-invariant subvariety $j : S \hookrightarrow \fr{g}$ which is open and dense in the support of $\scrA$ and such that $j^! \scrA$ is a local system $\scrE$ on $S$.

		By Lemma \ref{cusp_tech_lemma1}, there is some parabolic subalgebra $\fr{p}$ such that the costalks of the restriction for $x = s+n \in S$ can be computed via
		\begin{align*}
			 i_{x, \fr{l}}^! \LRes_P^G \scrA &= i_{x, \fr{g}}^! \scrA [2 \dim U] \\
			 &= i_{x, S}^! \scrE [ 2\dim U] \\
			&= i_{x,S}^! j^! \scrA [2\dim U] \\
			&= i_{x,S}^! \scrE [2\dim U]
		\end{align*}
		Since $\scrE$ is a nonzero local system on $S$, we have that $\LRes_P^G \scrA \neq 0$.
		But by assumption $\LRes_P^G \scrA = 0$ for any proper parabolic subalgebra $\fr{p}$, so we must have that $\fr{p} = \fr{g}$. 
		Recall that in the setup of Lemma \ref{cusp_tech_lemma1}, we had that $\fr{l} = Z_{\fr{g}} (s)$. 
		Since, $\fr{p} = \fr{g}$ and $\fr{g}$ is reductive, we must have that $\fr{l} = \fr{g}$. In particular, $Z_{\fr{g}} (s) = \fr{g}$ hence $s \in Z (\fr{g})$.
		As a consequence of $s$ being central, we have that $x \in Z(\fr{g}) + \scrN$. Therefore $S \subset Z(\fr{g}) + \scrN$, and likewise, $\overline{S} \subset Z(\fr{g}) + \scrN$.
	\end{proof}

	\begin{lemma}\label{supp_of_res}
		Let $P = L \ltimes U$ (with Lie algebras $\fr{p} = \fr{l} \oplus \fr{u}$) be a parabolic subgroup with Levi decomposition. 
		Suppose $\scrA \in \Perv_G (\fr{g}, \E)$. If $\supp (\scrA) \subseteq \scrN$, then $\supp (\LRes_P^G \scrA) \subseteq \scrN \cap \fr{l}$.
	\end{lemma}
	\begin{proof}
		A simple approximation of the support gives the following upper bound:
		\begin{align}
			\supp (\LRes_P^G \scrA) &= \supp (\LRes \scrA) \notag \\
			&= \supp (\For_L^P \circ \pi_! \circ i^* \circ \For_P^G \scrA) \notag \\
			&\subseteq \pi \left(\supp (i^* \For_P^G \scrA) \right) \tag{support of pushforward}\\
			&\subseteq \pi \left(  i^{-1} \left( \supp (\For_P^G \scrA )\right) \right)  \tag{support of pullback}\\
			&= \pi (i^{-1} (\scrN))
		\end{align}
		Note that $i^{-1} \scrN = \scrN \cap \fr{p} = (\scrN \cap \fr{l}) + \fr{u}$. Then $\pi (i^{-1} (\scrN)) = \scrN \cap \fr{l}$.
	\end{proof}

	\begin{lemma}\label{cusp_and_support_equiv}
		Let $\fr{g}$ be a semisimple Lie algebra. An irreducible perverse sheaf $\scrA \in \Perv_G (\fr{g}, \E)$ is cuspidal if and only if $\scrA$ and $\T_\fr{g} \scrA$ are supported in $\scrN$.
	\end{lemma}
	\begin{proof}
		By Lemma \ref{supp_of_res_small}, if $\scrA$ is res-small, then $\supp \scrA \subset Z(\fr{g}) + \scrN$.
		Since $\fr{g}$ is semisimple, $Z(\fr{g}) = 0$. Hence, $\supp \scrA \subset \scrN$. By Lemma \ref{cusp_iff_fourier_is_cuspidal}, $\T_{\fr{g}} \scrA$ is also cuspidal, so $\supp \T_{\fr{g}} \scrA \subset \scrN$ as well.

		Now suppose $\scrA$ and $\T_{\fr{g}} \scrA$ are supported on $\scrN$. For any parabolic with Levi decomposition, $\fr{p} = \fr{l} \oplus \fr{u}$.
		Define $\scrB = \LRes_P^G \scrA$ and $\T_{\fr{l}} \scrB = \T_{\fr{l}} \LRes_P^G \scrA = \LRes_P^G \T_{\fr{g}} \scrA$. By Lemma \ref{supp_of_res}, $\scrB$ and $\T_\fr{l} \scrB$ are both supported on $\scrN \cap \fr{l} \subseteq [\fr{l}, \fr{l}]$.

		Since $\scrB$ is supported on $[\fr{l}, \fr{l}]$, we can find a sheaf $\overline{\scrB}$ on $[\fr{l}, \fr{l}]$ such that
		\[ \scrB \cong (i_{ [\fr{l}, \fr{l}], \fr{l} })_* \overline{\scrB}\]
		We can view $\overline{\scrB}$ as a sheaf on $[\fr{l}, \fr{l}] \times \{0\} \rightarrow [\fr{l}, \fr{l}] \times Z(\fr{l})$, more precisely, we have that $\scrB \cong \overline{\scrB} \boxtimes (i_{0, \Z (\fr{l})})_* \underline{\E}_0$. 
		This formulation of $\scrB$ behaves well with regard to the Fourier transform. In particular,
		\begin{align}
			\T_\fr{l} \scrB &= \T_\fr{l} \left( \overline{\scrB} \boxtimes \left(i_{0,Z(\fr{l})}\right)_* \underline{\E}_0 \right) \notag \\
			&= \T_{[\fr{l}, \fr{l}]} \overline{\scrB} \boxtimes \T_{Z(\fr{l})} \left( i_{0, Z(\fr{l})}\right)_* \underline{\E}_0 \tag{external product compatibility} \\
			&= \T_{[\fr{l}, \fr{l}]} \overline{\scrB} \boxtimes \underline{\E}_{Z(\fr{l})} \notag
		\end{align}
		As a consequence, $\T_\fr{l} \scrB$ has $Z(\fr{l})$-invariant support. If $\fr{p}$ is a proper parabolic subalgebra, then $Z(\fr{l}) \neq 0$. 
		But by assumption $\T_\fr{l} \scrB$ has support in $\scrN$, so this former claim can only happen if $\LRes_P^G \T_{\fr{l}} \scrB = 0$ or equivalently, if $\LRes_P^G \scrB = 0$.
		This proves that $\scrA$ is res-small, and hence cuspidal.
	\end{proof}

	\begin{corollary}\label{semisimple_cusp_are_cs}
		Let $\fr{g}$ be semisimple. Then any cuspidal sheaf is a character sheaf and an orbital sheaf.
	\end{corollary}
	\begin{proof}
		Let $\scrA$ be cuspidal, then $\scrA$ and $\T (\scrA)$ are supported in $\scrN$. 
		Since there are only finitely many nilpotent orbits, it then follows that $\scrA$ is orbital and $\T (\scrA)$ is orbital. 
		By Corollary \ref{cs_are_fourier_of_orbital}, $\scrA$ is then also a character sheaf.
	\end{proof}

	\begin{lemma}\label{equiv_defn_for_cs_factorization}
		A simple sheaf $\scrA$ in $\Perv_G (\fr{g}, \E)$ is a character sheaf if and only if it is $\boxtimes$-factored with the central and semisimple factors being character sheaves.
	\end{lemma}
	\begin{proof}
		\textbf{[$\implies$]} First, suppose that $\scrA$ is a character sheaf. 
		Write $\scrA = \T_{\fr{g}} \scrB$ for an orbital sheaf $\scrB$. Let $\scrO_x$ be the $G$-orbit of some $x \in \fr{g}$ such that $\supp (\scrB) = \overline{\scrO_x}$.
		We can then write $x = z + y$ with $z \in Z (\fr{g})$ and $y \in [\fr{g}, \fr{g}]$. 
		Clearly, $\scrO_x = z + {}^G y$ and similarly $\overline{\scrO_x} = z + \overline{{}^G y}$. Let $a_z : \{z\} \hookrightarrow Z (\fr{g})$ be the obvious inclusion. 
		Therefore, $\scrB \cong a_{z*} \underline{\E}_{\{z\}} [\dim Z (\fr{g})] \boxtimes \scrB'$ where $\scrB'$ is a perverse sheaf on $[\fr{g}, \fr{g}]$ supported on $\overline{{}^G y}$.
		Since Fourier transforms commute with external tensor products, we have that $\scrA \cong \T (a_{z*} \underline{\E} \{z\}) [\dim Z (\fr{g})] \boxtimes \T (\scrB')$.
		Clearly, $\scrB'$ is orbital, and hence $\T (\scrB')$ is a character sheaf. Similarly, by Lemma \ref{abelian_cs:skyscaper_ls}, $\T (a_{z*} \underline{\E}_\{z\})$ is a character sheaf.

		\textbf{[$\impliedby$]} Let $\scrA = \scrL \boxtimes \scrB$ with $\scrL$ a character sheaf on $\Perv (Z(\fr{g}), \E)$ and $\scrB$ a character sheaf on $\Perv_G ([\fr{g}, \fr{g}], \E)$.
		By Corollary \ref{cs_are_fourier_of_orbital}, it suffices to show that $\T \scrA$ is orbital. Note that $\T \scrA \cong \T \scrL \boxtimes \T \scrB$. By assumption, $\T \scrL$ is a skyscraper sheaf and $\T \scrB$ is an orbital sheaf (since $\scrB$ is a character sheaf).
		Let $\scrO$ be a $G$-orbit in $[\fr{g}, \fr{g}]$ such that $\scrB$ is supported on $\overline{\scrO}$. We then have that $\T \scrA$ is supported on $\overline{\alpha + \scrO}$ for some $\alpha \in \fr{h}$. Finally, we remark that $\alpha + \scrO$ is a single $G$-orbit; therefore, $\T \scrA$ is orbital.
	\end{proof}
	
	\begin{midsecproof}{Theorem \ref{reductive_cusp_are_cs}}
		Let $\scrA$ be a cuspidal sheaf in $\Perv_G (\fr{g}, \E)$. We can write $\scrA = \scrL \boxtimes \scrB$ for $\scrL$ a character sheaf on $Z(\fr{g})$ and $\scrB$ a sheaf on $[\fr{g}, \fr{g}]$.
		By Lemma \ref{equiv_defn_for_cs_factorization}, it suffices to show that $\scrB$ is cuspidal, and hence a character sheaf by the semisimple case covered in Corollary \ref{semisimple_cusp_are_cs}.
		
		First, we note that any parabolic subalgebra $\fr{p}$ of $\fr{g}$ must contain $Z(\fr{g})$. 
		As a result, $\fr{p} = Z(\fr{g}) \oplus \fr{p}'$. Recall $i : \fr{p} \to \fr{g}$. We can then write it as $i = \id_{Z(\fr{g})} + i'$ where $i' : \fr{p}' \to [\fr{g}, \fr{g}]$.
		Similarly, for $\pi : \fr{p} \to \fr{l}$, we can write it as $\pi = \pi'' + \pi'$ for $\pi'' : Z(\fr{g}) \to \pi (Z(\fr{g}))$ and $\pi' : \fr{p}' \to \pi (\fr{p})$.
		Denote $\fr{l}' = \im \pi''$.
		We can now compute. We will omit the equivariance since it does not play a role. The action only occurs on $[\fr{g}, \fr{g}]$.
		\begin{align*}
			\LRes_P^G \left( \scrL \boxtimes \scrB \right) &= \pi_! \circ i^* (\scrL \boxtimes \scrB) \\
			&= \pi_! (\scrL \boxtimes (i')^* \scrB) \\
			&= (\pi_!'' \scrL \boxtimes \pi_!' (i')^* \scrB ) \\
			&= \pi_!'' \scrL \boxtimes \LRes_{L' \subset P'}^{[G, G]} \scrB
		\end{align*}
		Since $\scrL \neq 0$, and $\im \pi'' \neq 0$, we have that $\pi_!'' \scrL \neq 0$. As a result, we have that $\scrA$ is res-small if and only if $\scrB$ is also res-small.
		We note that the parabolic subalgebras $\fr{p}'$ give all the parabolic subalgebras in $[\fr{g}, \fr{g}]$. Therefore, $\scrB$ is a cuspidal sheaf.
	\end{midsecproof}

	\section{Admissible Sheaves}

	Classically, admissible sheaves were studied before Lusztig's discovery of character sheaves \cite{ICC}. They are easy to define, but in the positive characteristic case, there are many more than in the characteristic 0 case. Nonetheless, many of the same results can be proven true.
	Our goal will be to eventually prove that the class of admissible sheaves coincides with the class of character sheaves. This will make precise the earlier statement that character sheaves can be built from cuspidal sheaves. 
	Our strategy will be to reduce the study of admissible complexes to understanding how $\Ad G$-orbits behave under parabolic induction.

	\subsection{Admissible Sheaves}

	\begin{definition}\label{defn:admissible}
		Let $\Adm_G (\fr{g}, \E)$ be the full subcategory of $\Perv_G (\fr{g}, \E)$ generated by the irreducible subquotients of objects in the image of the functor
		\[ \Ind : \bigoplus_{\stackrel{P \subset G}{\textnormal{parabolic}}} \Cusp_P (\fr{p}, \E) \to \Perv_G (\fr{g}, \E) \]
		where $\Ind \scrA = \Ind_P^G \scrA$ for $\scrA \in \Cusp_P (\fr{p}, \E)$.
		The \textbf{admissible sheaves} are the simple objects in $\Adm_G (\fr{g}, \E)$.
	\end{definition}

	Our main result regarding admissible sheaves are that they agree with character sheaves. We will spend the rest of the section proving this equivalence.
	
	\begin{theorem}\label{admissible_and_cs_are_the_same}
		There is an equality of Serre subcategories of $\Perv_G (\fr{g}, \E)$,
		\[\Perv_G^{\CS} (\fr{g}, \E) = \Adm (\fr{g}, \E)\]
	\end{theorem}

	\begin{proposition}\label{res_and_ind_of_cs}
		Let $P = L \ltimes U$ be a proper parabolic subgroup of $G$ with Lie algebras $\fr{p} = \fr{l} \oplus \fr{u}$. 
		Let $\scrA \in \Perv_G (\fr{g}, \E)$ and $\scrB \in \Perv_L (\fr{l}, \E)$. Then the following holds
		\begin{enumerate}
			\item If $\scrA$ is orbital, then the composition factors of $\Res_P^G \scrA$ are orbital.
			\item If $\scrB$ is orbital, then the composition factors of $\Ind_P^G \scrB$ are orbital.
			\item If $\scrA$ is a character sheaf, then the composition factors of $\Res_P^G \scrA$ are character sheaves.
			\item If $\scrB$ is a character sheaf, then the composition factors of $\Ind_P^G \scrB$ are character sheaves.
		\end{enumerate}
	\end{proposition}
	\begin{proof}
		\textbf{(1): } We will only do the case of $\LRes$. The case of $\RRes$ can be obtained using similar methods. Let $\scrK$ be an irreducible subquotient of $\LRes \scrA$. By routine homological algebra arguments, we may assume without loss of generality, $\scrK$ is an irreducible subquotient of $\Ind \scrA_0$ with $\scrA_0$ irreducible and orbital.
		It suffices to show that $S = \supp (\LRes \scrA_0)$ is contained in the union of finitely many $L$-orbits of $\fr{l}$.
		Let $S_0 = \supp (\scrA_0)$. We can compute an estimate of the support $S$ of a restricted sheaf by direct means.
		\begin{align}
			S &= \supp (\LRes \scrA_0) \notag \\
			&= \supp (\For_L^P \circ \pi_! \circ i^* \circ \For_P^G \scrA_0) \notag \\
			&= \supp (\pi_! \circ i^* \circ \For_P^G \scrA_0) \notag \\
			&\subseteq \pi \left(\supp (i^* \For_P^G \scrA_0) \right) \tag{support of pushforward}\\
			&\subseteq \pi \left(  i^{-1} \left( \supp (\For_P^G \scrA_0 )\right) \right)  \tag{support of pullback}\\
			&= \pi (i^{-1} (S_0)) \notag \\
			&= \pi (S_0 \cap \fr{p})
		\end{align}
		It then suffices to show that $\pi (S_0 \cap \fr{p})$ is contained in the union of finitely many $L$-orbits in $\fr{l}$. 
		Since $\scrA_0$ is orbital, it is contained in the union of finitely many $G$-orbits on $\fr{g}$. 
		It then suffices to show that for any $G$-orbit $\scrO$ in $\fr{g}$, the set $\pi (\scrO \cap \fr{p})$ is the union of finitely many $L$-orbits in $\fr{l}$. 
		This is proved in Lemma 3.9 of \cite{Zhou2}.

		\textbf{(2): } Let $\scrK$ be an irreducible subobject of $\Ind \scrB$. By routine homological algebra arguments, we may assume without loss of generality, $\scrK$ is an irreducible subobject of $\Ind \scrB_0$ with $\scrB_0$ irreducible and orbital.
		It suffices to show that $S = \supp (\Ind \scrB_0)$ is contained in a union of finitely many $G$-orbits. 
		Let $S_0 = \supp (\scrB_0)$. We can compute an estimate of the support of an induced sheaf by direct means.
		\begin{align}
			S &= \supp (\Ind \scrB_0) \notag \\
			&= \supp (\Av_{P!}^G \circ i_! \circ \pi^* \circ \Av_{L!}^P \scrB_0) \notag \\
			 &\subseteq G \cdot \supp (i_! \pi^* \Av_{L!}^P \scrB_0) \tag{Equivariance} \\
			 &\subseteq G \cdot i\left(\supp (\pi^* \Av_{L!}^P \scrB_0 ) \right) \tag{Support of embedding} \\
			 &= G \cdot i\left(\pi^{-1} (\supp (\Av_{L!}^P \scrB_0 )) \right) \tag{Support of pullback }\\
			 &\subseteq G \cdot i \left(\pi^{-1} (S_0)\right) \tag{Equivariance }\\
			 &= G \cdot (\pi^{-1} (S_0)) \notag
		\end{align}
		Hence, $S$ is contained in the union of $G$-orbits of $\fr{g}$ which meet $\pi^{-1} (S_0)$.
		It will then suffice to show that $\pi^{-1} (S_0)$ is contained in a finite union of $G$-orbits in $\fr{g}$.
		Following the approach given in Proposition \ref{orbital_and_radon}, we only must check that the set $X^{\textnormal{ss}}$ of semisimple parts of the elements in $\pi^{-1} (S_0)$ is contained in a single $G$-orbit in $\fr{g}$.
		Since $\scrB_0$ is orbital, the set $X_0^{\textnormal{ss}}$ of semisimple parts of the elements in $S_0$ is contained in a single $L$-orbit of some $y$ in $\fr{l}$. We clearly have that $X^{\textnormal{ss}} \subseteq \pi^{-1} (X_0^{\textnormal{ss}})$. It then suffices to check that the set of semisimple elements in $\pi^{-1} (X_0^{\textnormal{ss}})$ is contained in a single $P$-orbit of $\fr{p}$.
		This is straightforward since for any element $x \in X_0^{\textnormal{ss}}$, we have that $\pi^{-1} (x) = x + \fr{u} = {}^U x$. As such, $\pi^{-1} (X_0^{\textnormal{ss}}) = \text{Ad}_U (X_0^{\textnormal{ss}})$. Thus,
		\[ \pi^{-1} (X^{\textnormal{ss}}_0) = \Ad_U (X^{\textnormal{ss}}_0) \subseteq \Ad_U (\Ad_L y) = G\cdot y\]
		which completes the proof of \textbf(2).

		\textbf{(3), (4): } These follow immediately after taking the Fourier transform of statements (1) and (2) (respectively).
	\end{proof}

	\begin{lemma}\label{admissible_tech_lemma}
		Let $P$ be a parabolic subgroup of $G$ with Levi decomposition $P = L \ltimes U$ with Lie algebras $\fr{p}, \fr{l}, \fr{u}$ respectively. 
		If $\scrB$ is an admissible sheaf in $\Adm_L (\fr{l}, \E)$, then any irreducible subobject $\scrA$ of $\Ind_P^G \scrB$ is admissible.
	\end{lemma}
	\begin{proof}
		Let $\scrA$ be an irreducible subobject of $\Ind_{L \subset P}^G \scrB$. 
		Since $\scrB$ is admissible, we can find a parabolic subgroup $Q$ of $L$ with Levi decomposition $Q = M \ltimes K$ (with Lie algebras $\fr{q} = \fr{m} \oplus \fr{k}$) and a cuspidal sheaf $\scrC$ on $\fr{m}$ such that $\scrB$ is an irreducible subobject of $\Ind_{K \subset Q}^L \scrB$

		From the $t$-exactness of $\Ind$, $\scrA$ is an irreducible subobject of $\Ind_P^G \Ind_Q^L \scrC$. 
		By Lemma \ref{ind_res_transitivity}, we have that $\scrA$ is an irreducible subobject of $\Ind_Q^G \scrC$, and hence $\scrA$ is admissible.
	\end{proof}

	\vspace{\topsep}
	\begin{midsecproof}{Theorem \ref{admissible_and_cs_are_the_same}}
		The proof that admissible sheaves are character sheaves follows immediately from Proposition \ref{res_and_ind_of_cs} and Theorem \ref{reductive_cusp_are_cs}.
		
		For the converse, we will provide a straightforward proof that models that of Mars and Springer, Theorem 9.3.2 \cite{MS}. 
		We argue by induction on the dimension of $\fr{g}$. The base case of $\dim \fr{g} = 0$ is obviously satisfied.
		If $\scrA$ is cuspidal, then there is nothing to show. By Lemma \ref{equiv_defn_for_cs_factorization}, we then may assume that there is some parabolic subgroup $P$ of $G$ such that $\RRes_P^G \scrA \neq 0$.
		There is then some irreducible subobject $\scrB \hookrightarrow \RRes_P^G \scrA$. By Proposition \ref{res_and_ind_of_cs}, we have that $\scrB$ is a character sheaf.
		Since $\RRes$ is a right adjoint to $\Ind$, we must have that
		\[ \Hom (\scrB, \LRes_P^G \scrA ) = \Hom (\Ind_P^G \scrB, \scrA) \neq 0\]
		where these hom-sets should be interpreted as between perverse sheaves. 
		Let $q : \Ind_P^G \scrB \to \scrA$ be any nonzero map. Since $\scrA$ is irreducible, we must have that $q$ is a quotient map.
		By induction, we can find a parabolic subgroup $Q$ of $L$ with Levi decomposition $Q = M \ltimes V$ (with Lie algebra $\fr{q} = \fr{m} \oplus \fr{v}$) and a cuspidal sheaf $\scrC$ on $\fr{m}$ such $\scrB$ is a subquotient of $\Ind_{Q}^P \scrC$.
		Let $P'$ be a parabolic subgroup in $G$ with Levi factor $M$. By Lemma \ref{ind_res_transitivity}, we have that $\Ind_P^G \scrB$ is a subquotient of $\Ind_{P'}^G \scrC$.
		Therefore, $\scrA$ is a composition factor of $\Ind_{P'}^G \scrC$ and hence is admissible.
	\end{midsecproof}

	\subsection{Lusztig Stratification}
	Let $\scrN$ be the nilpotent cone on $\fr{g}$. It is well known that $\scrN$ is closed under the $G$-action on $\fr{g}$. 
	Its orbits $\scrO_\lambda$ provide a stratification on $\scrN$. Let $P = L \ltimes U$ be a parabolic subgroup of $G$ with the corresponding Levi subgroup $L$. Let $\fr{l}$ be the Lie algebra corresponding to $L$.
	Denote the regular part of the center $Z(\fr{l})$ by $Z_r (\fr{l})$.

	\begin{definition}\label{def:lusztig_strata}
		Let $L$ be a Levi subgroup of $G$ with Lie algebra $\fr{l}$. Let $\scrO$ be a nilpotent orbit of $\fr{l}$. Define locally closed varieties,
		\[S_{\fr{l}, \scrO} = {}^G (Z_r (\fr{l}) + \scrO) \]
		where $Z_r (\fr{l})$ denotes the regular part of the center $Z (\fr{l})$, and is given by $Z_r (\fr{l}) = \{ x\in \fr{l}, Z_g (x) = \fr{l} \}$. 
		We call the subvarieties $S_{\fr{l}, \scrO}$ the \textbf{Lusztig strata} of $\fr{g}$. 
	\end{definition}

	Let $\Lambda$ be the collection of pairs $(\fr{l}, \scrO)$ where $\fr{l}$ is a Levi subalgebra and $\scrO$ is a nilpotent orbit of $\fr{l}$. There is a natural action of $G$ on $\Lambda$ by conjugation of both factors.

	\begin{lemma}[\cite{Lus95}]\label{lem:lusztig_strata_is_stratification}
		The Lusztig strata provide a stratification for $\fr{g}$. In particular,
		\[ \fr{g} = \bigsqcup_{(\fr{l}, \scrO) \in \Lambda/G} S_{\fr{l}, \scrO}\]
	\end{lemma}

	It is useful to define varieties related to a Lusztig strata. Let $\fr{u}_P$ denote the Lie algebra of the unipotent radical of $P$. We set
	\[\tilde{S}_{\fr{l}, \scrO} = G \times^L (Z_r (\fr{l}) + \scrO),  \]
	\[T_{\fr{l}, \scrO} = {}^G (Z_r (\fr{l}) + \scrO + \fr{u}_P), \hspace{1cm} \tilde{T}_{\fr{l}, \scrO} = G \times^P (Z_r (\fr{l}) + \scrO + \fr{u}_P ).\]
	If the choice of $G$ is not clear, we may also write $S_{\fr{l}, \scrO}^G$ for $S_{\fr{l}, \scrO}$ (and likewise for the other related varieties).
	One can observe that $\overline{S_{\fr{l}, \scrO}} = T_{\fr{l}, \scrO}$. The adjoint action of $G$ on $\fr{g}$ gives rise to morphisms,
	\[\varpi_{\fr{l}, \scrO} : \tilde{S}_{\fr{l}, \scrO} \to S_{\fr{l}, \scrO}, \hspace{1cm} \pi_{\fr{l}, \scrO} : \tilde{T}_{\fr{l}, \scrO} \to T_{\fr{l}, \scrO}.\]
	The natural morphism $G \times^L \fr{l} \to G \times^P \fr{p}$ gives rise to an open embedding $\tilde{i} : \tilde{T}_{\fr{l}, \scrO} \to \tilde{S}_{\fr{l}, \scrO}$.
	These varieties and their maps fit into the following cartesian square:

	{ \centering
	\begin{tikzcd}
		{\tilde{T}_{\fr{l}, \scrO}} \arrow[r, "\tilde{i}", hook] \arrow[d, "{\varpi_{\fr{l}, \scrO}}"'] & {\tilde{S}_{\fr{l}, \scrO}} \arrow[d, "{\pi_{\fr{l}, \scrO}}"] \\
		{T_{\fr{l}, \scrO}} \arrow[r, "i", hook]                                                        & {S_{\fr{l}, \scrO}}                                           
		\end{tikzcd}
	\par }

	For a pair $(\fr{l}, \scrO) \in \Lambda$, let $j_{\fr{l}, \scrO}: Z_r (\fr{l}) + \scrO \hookrightarrow \fr{g}$ and $k_{\fr{l}, \scrO} : S_{\fr{l}, \scrO} \hookrightarrow \fr{g}$ denote the obvious inclusions.

	In the ordinary theory of character sheaves over reductive groups \cite{CS1} or reductive Lie algebras \cite{M}, the Lusztig stratification provides a sufficient stratification for describing the character sheaves.
	The benefit of this perspective on character sheaves is that it provides a concrete way to parameterize the character sheaves and reduces the work of describing them to explaining the topology of the strata.

	A stronger version of constructibility along Lusztig strata, called quasi-admissibility, was introduced by Mirković \cite{M}. We will use this formulation to prove constructibility.

	\begin{definition}\label{def:quasiadmissible}
		A perverse sheaf $\scrA \in \Perv_G (\fr{g}, \E)$ is said to be \textbf{quasi-admissible} if for all $(\fr{l}, \scrO) \in \Lambda$ the simple constituents of $j_{\fr{l}, \scrO}^! \scrA$ are of the form $\scrL\vert_{Z_r (\fr{l})} \boxtimes \scrE$ where $\scrL$ is a character sheaf on $Z (\fr{l})$ and $\scrE \in \Loc_L (\scrO, \E)$.
	\end{definition}

	\begin{lemma}\label{cusp_tech_lemma2}
		Let $P = L \ltimes U$ be a parabolic subgroup with Levi decomposition (with Lie algebras $\fr{p} = \fr{l} \oplus \fr{u}$). Pick a nilpotent orbit $\scrO$ of $\fr{l}$.
		Let $j_{\fr{l}, \scrO} : Z_r (\fr{l}) + \scrO \to \fr{l}$ denote the obvious inclusion.
		Then for any $\scrA \in \Perv_G (\fr{g}, \E)$,
		\[ (j_{\fr{l},\scrO}')^! \RRes_P^G \scrA = j_{\fr{l},\scrO}^! \scrA [2\dim U].\]
	\end{lemma}
	\begin{proof}
		The proof will follow closely that of Lemma \ref{cusp_tech_lemma1}.
		Consider the commutative diagram where the center square is cartesian:

		{ \centering
		\begin{tikzcd}
			& Z_r (\fr{l}) + \scrO \arrow[dd, "j'"] \arrow[rd, "k", hook, bend left] \arrow[rrdd, "{j_{\fr{l}, \scrO}}", hook, bend left=49] \arrow[ld, "{j_{\fr{l}, \scrO}'}"', hook] &                                                                                                                                           &                           \\
\fr{l} \arrow[rd, "j"'] &                                                                                                                                                                          & Z_r (\fr{l}) + \scrO + \fr{u} = {}^U (Z_r (\fr{l}) + \scrO) \arrow[lu, "a", two heads, bend left] \arrow[rd, "{j_{\fr{l}, \scrO}''}", hook] &                           \\
			& \fr{g}/\fr{u}                                                                                                                                                            &                                                                                                                                           & \fr{g} \arrow[ll, "\tau"]
\end{tikzcd}
		\par }

		Note that $Z_r (\fr{l}) + \scrO + \fr{u} = {}^U (Z_r (\fr{l}) + \scrO)$ follows from Lemma \ref{orbit_calc_lemma} (1).

		We can now perform a proper base change to obtain
		\begin{align*}
			(j_{\fr{l}, \scrO}')^! \RRes_P^G \scrA &\cong (j_{\fr{l}, \scrO}')^! j^! \tau_* \scrA \\
			&\cong (j')^!  \tau_* \scrA \\
			&\cong  a_* (j_{\fr{l}, \scrO}'')^! \scrA
		\end{align*}

		Since $(j_{\fr{l}, \scrO}'')^! \scrA$ is $U$-equivariant, it is locally constant on $U$-orbits. And so since $U$ is unipotent, we have an isomorphism
		\[ (j_{\fr{l}, \scrO}'')^! \scrA \cong a^! k^! (j_{\fr{l}, \scrO}'')^! \scrA \cong a^! j_{\fr{l}, \scrO}^! \scrA \].
		Finally, applying $a_*$ to both sides of the above equation finishes the proof.
		\begin{align*}
			a_{*} (j_{\fr{l}, \scrO}'')^! \scrA &\cong a_* a^! j_{\fr{l}, \scrO}^! \scrA \\
			&\cong a_{*} a^* j_{\fr{l}, \scrO}^! \scrA [2\dim U] \\
			&\cong j_{\fr{l}, \scrO}^! \scrA [2\dim U]
		\end{align*}
		where $a_{*} a^* \cong \text{Id}$ follows from $\fr{u}$ being contractible.
	\end{proof}

	\begin{lemma}\label{lem:cs_are_qad}
		Every character sheaf $\scrA$ is quasi-admissible.
	\end{lemma}
	\begin{proof}
		Let $(\fr{l}, \scrO) \in \Lambda$. Let $\fr{p}$ be a parabolic subalgebra such that $\fr{l}$ is its Levi subalgebra. By Lemma \ref{cusp_tech_lemma2}, one has that
		\[j_{\fr{l}, \scrO}^! \RRes_{L \subset P}^G \scrA \cong j_{\fr{l}, \scrO}^! \scrA [2\dim U]\]

		By Proposition \ref{res_and_ind_of_cs}, we have that $\RRes_{L \subset P}^G \scrA \in \Perv_G^{\CS} (\fr{g}, \E)$. 
		Let $\scrA'$ be a simple constituent of $\RRes_{L \subset P}^G \scrA$. 
		By Lemma \ref{equiv_defn_for_cs_factorization}, one can $\boxtimes$-factor $\scrA'$ as $\scrA' \cong \scrL \boxtimes \scrB$ where $\scrL$ is a character sheaf on $Z(\fr{l})$ and $\scrB$ is a character sheaf on $[\fr{l}, \fr{l}]$.
		Let $j_{\scrO} : \scrO \hookrightarrow \fr{l}$ denote the inclusion. We then can $!$-pullback the $\boxtimes$-factorization to get
		\[j_{\fr{l}, \scrO}^! \scrA' = \scrL\vert_{Z_r(\fr{l})} \boxtimes j_{\scrO}^! \scrB\]
		since $Z_r (\fr{l})$ is open in $Z (\fr{l})$. Note that $\scrB$ is constructible with respect to the $G$-orbits on $[\fr{l}, \fr{l}]$, so we must have that $j_{\scrO}^! \scrB$ is a local system.
	\end{proof}

	As a consequence of Lemma \ref{lem:cs_are_qad}, we can see that character sheaves are constructible with respect to the Lusztig strata. In particular, every character sheaf arises as an IC-extension from them.
	In a similar vein, one might ask what Lusztig strata and local systems are allowed to give rise to cuspidal sheaves. The following lemma gives such a characterization.

	\begin{lemma}\label{lem:constr_of_cuspidal}
		Let $\scrC$ be a cuspidal sheaf on $\fr{g}$. 
		Then $\scrC = \IC (S_{\fr{g}, \scrO}, \scrE)$ for some distinguished nilpotent orbit $\scrO$ of $\fr{g}$ and some irreducible local system $\scrE \in \Loc_G (S_{\fr{g}, \scrO})$ such that $\scrE \cong \scrL^\circ \boxtimes \scrE'$ where $\scrL^\circ \in \Loc (Z_r (\fr{g}))$ and $\scrE' \in \Loc_G (\scrO)$. 
	\end{lemma}
	\begin{proof}
		The statement of the lemma is well-defined since $S_{\fr{g}, \scrO} = Z_r (\fr{g}) + \scrO$. 
		Since $\scrC$ is cuspidal, we can $\boxtimes$-factor it as $\scrC \cong \scrL \boxtimes \scrB$ where $\scrL$ is a character sheaf on $Z (\fr{g})$ and $\scrB$ is a sheaf on $[\fr{g}, \fr{g}]$.
		As we saw in the proof of Theorem \ref{reductive_cusp_are_cs}, $\scrB$ is also cuspidal. Hence, by Lemma \ref{cusp_and_support_equiv}, we have that $\scrB$ is supported on the nilpotent cone $\scrN$.
		Since $\scrB$ is irreducible, we must have that 
		\begin{equation}
			\scrB \cong \IC (\scrO,\scrE') \label{eq:qad_cs_and_cuspidal_1}
		\end{equation}
		where $\scrO$ is a distinguished nilpotent orbit in $\fr{g}$ and $\scrE' \in \Loc_G (\scrO)$.
		Also, one can observe that since $Z (\fr{g})$ is smooth and $Z_r (\fr{g})$ is dense in $Z (\fr{g})$, we have that 
		\begin{equation}
			\scrL \cong \IC (Z_r (\fr{g}), \scrL\vert_{Z_r (\fr{g})} [-\dim Z(\fr{g})]). \label{eq:qad_cs_and_cuspidal_2}
		\end{equation}
		We can then denote $\scrL^\circ = \scrL\vert_{Z_r (\fr{g})} [-\dim Z_(\fr{g})]$. 
		By combining the observations in \ref{eq:qad_cs_and_cuspidal_1} and \ref{eq:qad_cs_and_cuspidal_2}, we obtain an IC-description for $\scrC$:
		\[\scrC \cong \IC (S_{\fr{g}, \scrO}, \scrL^\circ \boxtimes \scrE').\]
	\end{proof}

	\subsection{Generation of Admissible Sheaves}

	In the definition of admissible sheaves, we take all composition factors of the parabolic induction of cuspidal sheaves. However, there may be significant overlap in the composition factors of the parabolic induction of two cuspidal sheaves. 
	In this section, we will propose an alternate definition of admissible that we will later see is in some sense minimal (Theorem \ref{decomp_of_ad_sheaves}).

	\begin{lemma}\label{lem:subobjs_suffices}
		If $\scrA$ is admissible, then there exists some Levi subgroup $L$ with parabolic $P$ and cuspidal sheaf $\scrC$ on $\fr{l}$ such that $\scrA$ is a quotient (alternatively, a subobject) of $\Ind_{P}^G \scrC$.
	\end{lemma}
	\begin{proof}
		We will argue only the statement regarding quotients. The subobject statement can be derived by applying the obvious modifications. 
		Define a set
		\[S(\scrA) = \{ L \subseteq G \text{ Levi subgroup }\mid \RRes_{L \subseteq P}^G \scrA \neq 0  \text{ for some parabolic } P \text{ containing } L\}.\]
		We can give $S(\scrA)$ a partial order by inclusion of Levi subgroups.

		Let $L \in S(\scrA)$ be minimal with parabolic $P$ such that $\RRes_{L \subseteq P}^G \scrA \neq 0$. 
		We can then pick some $\scrC \in \Perv_L (\fr{l}, \E)$ such that $\scrC \hookrightarrow \RRes_{L \subseteq P}^G \scrA$.
		Let $Q$ be any parabolic subgroup of $L$ with Levi $M$. We can then find parabolic subgroup $P'$ in $G$ with Levi $M$ such that $P' \cap P = Q$. Now by transitivity of induction,
		\[\RRes_{M \subseteq Q}^L \RRes_{L \subset P}^G \scrA \cong \RRes_{M \subseteq P'}^G \scrA = 0\]
		by minimality of $L \in S(\scrA)$. In particular, $\scrC$ is res-small. 
		By Proposition \ref{res_and_ind_of_cs}, we have that $\scrC$ is a character sheaf; therefore, $\scrC$ is cuspidal.

		By adjunction, we then have a surjection $\Ind_{L \subseteq P}^G \scrC \twoheadrightarrow \scrA$. Since $\scrA$ is simple, this map is a quotient as desired.
	\end{proof}

	Lemma \ref{lem:subobjs_suffices} gives an alternative definition for admissible sheaves. 
	Namely, $\scrA \in \Perv_G (\fr{g}, \E)$ is admissible if and only if there exists some Levi subgroup $L$ (with Lie algebra $\fr{l})$ and a cuspidal sheaf $\scrC$ on $\fr{l}$ such that $\scrA$ is a subobject (alternatively, quotient) of $\Ind \scrC$.
	
	As a result, if $\scrA$ is admissible, we can find Levi subgroups $L_1, L_2$ (with Lie algebras $\fr{l}_1, \fr{l}_2$, respectively) and cuspidal sheaves $\scrC_1 \in \Perv_{L_1} (\fr{l}_1, \E)$ and $\scrC_2 \in \Perv_{L_2} (\fr{l}_2, \E)$ such that $\scrA$ is a subobject of $\Ind \scrC_1$ and a quotient of $\Ind \scrC_2$.
	A priori, it may not be the case that the Levis and cuspidals can be chosen such that $L_1 = L_2$ and $\scrC_1 = \scrC_2$. 
	We will spend the remainder of the section proving that we can make such a choice.

	\begin{proposition}\label{prop:subobjs_equals_subquots}
		Let $P = L \ltimes U$ be a parabolic subgroup with Levi decomposition. Let $\fr{p} = \fr{l} \oplus \fr{u}$ denote the corresponding Lie algebras.
		If $\scrC \in \Perv_L (\fr{l}, \E)$ is cuspidal, then the set of isomorphism classes of subobjects of $\Ind_P^G \scrC$ and the set of isomorphism classes of quotients of $\Ind_P^G \scrC$ are equal.
	\end{proposition}

	In order to prove this proposition, we will first need to develop an understanding of when $\Ind_P^G \scrC$ is an IC-extension of a local system.
	Fix a cuspidal sheaf $\scrC \in \Perv_L (\fr{l}, \E)$. By Lemma \ref{lem:constr_of_cuspidal}, we have that $\scrC \cong \IC (S_{\fr{l}, \scrO}^L, \scrE)$ for some irreducible $\scrE \in \Loc_L (S_{\fr{l}, \scrO})$. 
	Define a group,
	\[ N_G (L, \scrO) = \{ n \in N_G (L) \mid n \cdot \scrO = \scrO \}.\]

	\begin{lemma}[\cite{Let05}, proof of Lemma 5.1.28]\label{lem:varpi_is_galois_covering}
		The morphism $\varpi_{\fr{l}, \scrO} : \tilde{S}_{\fr{l}, \scrO} \to S_{\fr{l}, \scrO}$ is a Galois covering with Galois group $N_G (L, \scrO)/L$.
	\end{lemma}

	Consider the quotient map $q : G \times (Z_r (\fr{l}) + \scrO) \to \tilde{S}_{\fr{l}, \scrO}^G$.
	There is a unique $\tilde{\scrE} \in \Loc_G (S_{\fr{l}, \scrO}^G)$ such that $q^* \tilde{\scrE} \cong \underline{\E}_G \boxtimes \scrE$.
	By Lemma \ref{lem:varpi_is_galois_covering}, we have that $\varpi_{\fr{l}, \scrO *} \tilde{\scrE}$ is a local system on $S_{\fr{l}, \scrO}^G$.

	\begin{lemma}\label{lem:when_is_ind_an_IC}
		With the same notation as above, there exists a canonical isomorphism,
		\[\Ind_{L \subseteq P}^G \IC (S_{\fr{l}, \scrO}^L, \scrE) \cong \IC (S_{\fr{l}, \scrO}^G, \varpi_{\fr{l}, \scrO *} \tilde{\scrE})\]
	\end{lemma}
	\begin{proof}
		This result is stated for $\scrE \cong \underline{\E}_{Z_r (\fr{l})} \boxtimes \scrE'$ for $\scrE' \in \Loc_L (\scrO)$ in Proposition 2.17 of \cite{AHJR1} and more generally (in the characteristic 0 setting) in Proposition 5.1.33 \cite{Let05}.
		There is no obstacle in generalizing Proposition 2.17 of \cite{AHJR1} to allow for more general local systems on $Z_r (\fr{l})$. Nonetheless, we will go over the key ideas. First, we will use the description of $\Ind_P^G$ as given in Lemma \ref{lem:ind_equiv_version_of_ind}. 
		Let $\overline{q} : G \times (Z_r (\fr{l}) + \scrO + \fr{u}_P) \to \tilde{T}_{\fr{l}, \scrO}^G$ be the obvious quotient map.
		There is a unique local system $\hat{\scrE} \in \Loc_G (\tilde{T}_{\fr{l}, \scrO}^G)$ such that $\overline{q}^* \scrE \cong \underline{\E}_G \boxtimes \scrE \boxtimes \underline{\E}_{\fr{u}_P}$.
		Basic sheaf functor compatibilities yields the following:
		\begin{align*}
			\Ind_{L \subseteq P}^G \IC (S_{\fr{l}, \scrO}^L, \scrE) &\cong \mu_! I_P^G \pi^* \IC (S_{\fr{l}, \scrO}^L, \scrE) [2 \dim G/P] \\
			&\cong \pi_! I_P^G \IC (Z_r (\fr{l}) + \scrO + \fr{u}_P, \scrE \boxtimes \underline{\E}_{\fr{u}_P}) [2 \dim G/P- \dim \fr{u}_P] \\
			&\cong \pi_! \IC (T_{\fr{l}, \scrO}^G, \hat{\scrE})
		\end{align*}
		The last part is showing that
		\[\pi_! \IC (T_{\fr{l}, \scrO}^G, \hat{\scrE}) \cong \IC (S_{\fr{l}, \scrO}^G , \varpi_{\fr{l}, \scrO*} \tilde{\scrE}). \label{eq:when_is_ind_an_IC_1}\]
		By base change and the uniqueness of the defining conditions of $\hat{\scrE}$ and $\tilde{\scrE}$ it is straightforward to show that $(pi_! \IC (T_{\fr{l}, \scrO}^G, \hat{\scrE}))\mid_{S_{\fr{l}, \scrO}^G} \cong \varpi_{\fr{l}, \scrO*} \tilde{\scrE} [\dim S_{\fr{l}, \scrO}^G ]$.
		To prove that the left-hand side of \ref{eq:when_is_ind_an_IC_1} is indeed an IC-extension, one must check the conditions which uniquely characterize IC-sheaves [\cite{Let05}, 4.1.1].
		There is no obstacle in translating the characteristic 0 proof that these conditions hold (see \cite{Let05}, Proposition 5.1.33) into the modular setting.
		Further discussion on checking these conditions can be found in the proof of Proposition 2.17 of \cite{AHJR1}.
	\end{proof}

	\begin{midsecproof}{Proposition \ref{prop:subobjs_equals_subquots}}
		By Lemma \ref{lem:when_is_ind_an_IC}, and since $\IC$ preserves subobjects and quotients (Lemma 3.3.5, \cite{A1}), it suffices to show that the set of isomorphism classes of subobjects of $\varpi_{\fr{l}, \scrO *} \tilde{\scrE}$ and the set of isomorphism classes of quotients of $\varpi_{\fr{l}, \scrO *} \tilde{\scrE}$ are equal.
		Equivalently, for $\scrL \in \Loc_G (S_{\fr{l}, \scrO}^G)$ simple, we must show that
		\[\dim \Hom_{\Loc_G (S_{\fr{l}, \scrO})} (\varpi_{\fr{l}, \scrO*} \tilde{\scrE}, \scrL) = \dim \Hom_{\Loc_G (S_{\fr{l}, \scrO})} (\scrL, \varpi_{\fr{l}, \scrO*} \tilde{\scrE}) \label{eq:subobjs_equals_subquots_1}\]
		By \ref{lem:varpi_is_galois_covering}, we know that $\varpi_{\fr{l}, \scrO}$ is a Galois covering so $\varpi_{\fr{l}, \scrO*}$ and $\varpi_{\fr{l}, \scrO}^*$ are biadjoint. If we apply biadjointness to \ref{eq:subobjs_equals_subquots_1}, we obtain the following equivalent equation
		\[\dim \Hom_{\Loc_G (S_{\fr{l}, \scrO})} (\tilde{\scrE}, \varpi_{\fr{l}, \scrO}^* \scrL) = \dim \Hom_{\Loc_G (S_{\fr{l}, \scrO})} (\varpi_{\fr{l}, \scrO}^* \scrL, \tilde{\scrE}) \label{eq:subobjs_equals_subquots_2}.\]
		In the proof of Lemma 2.3 of \cite{AHJR2}, it is shown that $\varpi_{\fr{l}, \scrO}^* \scrL$ is semisimple. This observation along with the fact that $\tilde{\scrE}$ is simple shows that \ref{eq:subobjs_equals_subquots_2} holds.
	\end{midsecproof}

	\subsection{Examples}

	\begin{example}
	We will work out the character sheaf theory for $\sl_2 (\C)$ (with the analytic topology) and some of the related objects such as orbital sheaves and cuspidal sheaves.
	For the sake of notation, all sheaves will be in $\E$-coefficients with $\text{char} (\E) = \ell$. 
	
	\noindent\textbf{Geometry of Orbits}

	Let $\fr{g}= \sl_2 (\C)$ and $G = \SL_2 (\C)$. We will first start with the geometry of $G$-orbits and the orbital sheaves.
	Let $\lambda \in \C$, there is a semisimple orbit $\scrO_\lambda$ given by the $G$-orbit of $\begin{bmatrix} \lambda & 0 \\ 0 & -\lambda \end{bmatrix}$.
	This is nearly a bijection between $\C$ and the semisimple orbits, namely, there is a bijection
	\[ \lambda \leftrightarrow \scrO_\lambda \hspace{1cm} \C/{\pm 1} \leftrightarrow \{ \text{semisimple orbits of } G \} \]
	There is only one orbit of $G$ which is not semisimple, it consists of the unique nonzero nilpotent orbit $\scrQ$ given by the $G$-orbit of $\begin{bmatrix} 0 & -1 \\ 0 & 0\end{bmatrix}$. 

	We can summarize the possible local systems on strata as follows:
	\[\Loc_G (\scrO_\lambda) \cong \E\text{-mod} \]
	\[\Loc_G (\scrQ) \cong \E [\Z/2]\text{-mod}\]

	\noindent\textbf{Cuspidal Sheaves}

	By Lemma \ref{cusp_and_support_equiv}, the only possible cuspidal sheaves are IC-complexes on the nilpotent orbits.
	Note that $\IC (\scrO_0, \E)$ is never cuspidal, since its Fourier transform is the constant sheaf $\underline{\E}_{\fr{g}}$ which does not have nilpotent support.

	\textit{Case 1: $\ell \neq 2$:} There are two irreducible local systems on $\scrQ$: the constant sheaf $\underline{\E}_{\scrQ}$ and local system $\scrL_{\text{sgn}}$ corresponding to the sign representation of $\Z/2\Z$. 
	Now consider $\IC (\scrQ, \E)$. 
	he Springer sheaf $\Ind_B^G \delta_0$ decomposes as a direct sum of the skyscraper sheaf $\IC (\scrO_0, \E)$ and $\IC (\scrQ, \E)$. In particular, $\IC (\scrQ, \E)$ is not cuspidal.

	Note that the Grothendieck-Springer sheaf $\Ind_B^G \underline{\E}_T [1]$ decomposes as a direct sum of the constant sheaf $\underline{\E}[3]$ and $\IC (\fr{g}_{\text{rs}}, \scrL_{\text{sgn}}')$ where $\scrL_{\text{sgn}}'$ is the irreducible local system on $\fr{g}_{\text{rs}}$ coming from the irreducible nontrivial representation of $W \cong \Z/2\Z$.
	By exhaustion of the sheaves induced from cuspidal sheaves on the torus, we then see that $\IC (\scrQ, \scrL_{\text{sgn}})$ is cuspidal.
	
	\textit{Case 2: $\ell = 2$:} The only irreducible local system on $\scrQ$ is the constant sheaf. 
	The Springer sheaf no longer decomposes as a direct sum, and $\IC (\scrQ, \E)$ is no longer a subobject/quotient.
	Similarly, the Grothendieck-Springer sheaf only has the constant sheaf as a subobject/quotient.
	As a result, $\IC (\scrQ, \E)$ is the unique cuspidal sheaf. This can be seen in another way. 
	Namely, when $\ell = 2$, $\T \IC (\scrQ, \E) \cong \IC (\scrQ, \E)$. As a result, $\IC (\scrQ, \E)$ could also be seen to be cuspidal by Lemma \ref{cusp_and_support_equiv}.

	\noindent\textbf{Orbital Sheaves}

	Since we are working in the analytic setting, we only allow for $\G_m$-equivariant sheaves. Since any $\G_m$-orbit of a semisimple element passes through infinitely many semisimple orbits, all orbital sheaves are supported on the nilpotent cone.

	\textit{Case 1: $\ell \neq 2$:}
	 The orbital sheaves are given by $\IC(\scrO_0, \E), \IC(\scrQ, \E)$, and $\IC(\scrQ, \scrL_{\text{sgn}})$.

	\textit{Case 2: $\ell = 2$: }
	The orbital sheaves are given by $\IC(\scrO_0, \E)$ a nd $\IC (\scrQ, \E)$.

	\noindent\textbf{Character Sheaves}

	We can leverage Corollary \ref{cs_are_fourier_of_orbital} to obtain the character sheaves by Fourier transform. 
	Note that $\T (\IC(\scrO_0, \E)) \cong \underline{\E}_{\fr{g}} [\dim \fr{g}] = \IC (\fr{g}_{\text{rs}}, \E)$ is always a character sheaf.

	\textit{Case 1: $\ell \neq 2$:}
	Since $\IC (\scrQ, \scrL_{\text{sgn}})$ is cuspidal, it is a character sheaf. We can also see this by Fourier transform. Namely, since the Fourier transform of $\IC (\scrQ, \scrL_{\text{sgn}})$ has nilpotent support, we can deduce that $\T \IC (\scrQ, \scrL_{\text{sgn}}) \cong \IC (\scrQ, \scrL_{\text{sgn}})$. 
	Since $\IC (\scrQ, \E)$  is not cuspidal, its support cannot be contained in $\scrN$. 
	By process of elimination, $\T \IC (\scrQ, \E) \cong \IC (\fr{g}_{\text{rs}}, \scrL_{\text{sgn}}')$.

	\textit{Case 2: $\ell = 2$: } 
	Since there are two orbital sheaves, there are two character sheaves. One of them is covered by the constant sheaf as described by the above. The other character sheaf must be the unique cuspidal sheaf $\IC (\scrQ, \E)$.

	\noindent\textbf{Étale Case}

	We will briefly explain some of the differences between the étale case and the analytic case. The classification of cuspidal sheaves remains the same provided $\text{char} (k) = p$ is sufficiently large.
	A new feature appears in the form of new orbital sheaves, and hence, new character sheaves. One still has in the étale case that the semisimple orbits are in bijection with $\overline{k}/\pm 1$.
	One can then consider the IC-complexes, $\scrE_\lambda = \IC (\scrO_\lambda, \E)$ where $\lambda \in \overline{k}/\pm 1$. Clearly, none of these IC-complexes are cuspidal (nor are their Fourier transforms) since their support is not contained in the nilpotent cone.
	Nonetheless, $\T (\scrE_\lambda)$ are character sheaves.
		
	\end{example}

	\section{Singular Support Characterization}

	Character sheaves were characterized by their singular support independently by Mirković and Vilonen \cite{MV} and Ginzburg \cite{G}. 
	Both of their results crucially used varieties of characteristic 0. In the Lie algebra setting, Mirković \cite{M} also characterized character sheaves on Lie algebras in terms of their singular support (along with a constructibility requirement).
	This worked in the setting of $D$-modules. The main issue with extending their results was the lack of a good notion of singular support for varieties of positive characteristic.
	
	Beilinson, in \cite{Bei}, generalized the notion of singular support for constructible sheaves defined over any characteristic. Since then the above singular support characterizations have found positive characteristic variants. Namely, Psarmoligkos \cite{P} proved that character sheaves on a reductive group have nilpotent singular support.
	Shortly after Zhou in \cite{Zhou2}, proved Mirković's singular support characterization of character sheaves on Lie algebras using similar methods. 
	The methods of Zhou should presumably extend to the modular case as well. We will work through the details for one direction of the implication. 

	\begin{proposition}\label{thm:sing_supp_and_cs}
		Let $\scrA$ be a simple perverse sheaf in $\Perv_G (\fr{g}, \E)$. If $\scrA$ is a character sheaf, then $SS(\scrA) \subset \fr{g} \times \scrN_{\fr{g}}$.
	\end{proposition}

	We will closely follow the proof of Theorem 5.1 in \cite{Zhou2}. 

	\begin{lemma}\label{semisimple_cusp_nilp_ss}
		Suppose $\fr{g}$ is semisimple. Let $\scrC \in \Perv_G (\fr{g}, \E)$ be a simple perverse sheaf. Then $\scrC$ is cuspidal if and only if $SS(\scrC) \subseteq \fr{g} \times \scrN_{\fr{g}}$.
	\end{lemma}
	\begin{proof}
		By Lemma 8.2.7 (3) of \cite{A1}, $\scrC$ is $\G_m$-equivariant for the weight 2 action. By \cite{Zhou1}, $CC(\scrC) = CC(\T \scrC)$, and hence $SS(\scrC) = SS(\T \scrC)$.
		In particular, $SS(\scrC)$ is nilpotent if and only if $\supp (\scrC) \subseteq \scrN_{\fr{g}}$ and $\supp (\T \scrC) \subseteq \scrN_{\fr{g}}$.
		By Lemma \ref{cusp_and_support_equiv}, we have that $SS(\scrC) \subseteq \fr{g} \times \scrN_{\fr{g}}$ if and only if $\scrC$ is cuspidal.
	\end{proof}

	\begin{lemma}\label{red_cusp_nilp_ss}
		If $\scrC \in \Perv_G (\fr{g}, \E)$ is cuspidal, then $SS (\scrC) \subseteq \fr{g} \times \scrN_{\fr{g}}$.
	\end{lemma}
	\begin{proof}
		We can find a $\boxtimes$-factorization of $\scrC$ as $\scrC \cong \scrL \boxtimes \scrB$. 
		Let $\pr_2 : Z(\fr{g}) \times [\fr{g}, \fr{g}] \to [\fr{g}, \fr{g}]$ be the projection. One has that 
		\[SS (\scrC) = SS (\scrL \boxtimes \scrB) = \pr_2^\circ SS(\scrB)\]
		As we saw in the proof of Theorem \ref{reductive_cusp_are_cs}, $\scrB$ is cuspidal.
		So by Lemma \ref{semisimple_cusp_nilp_ss}, we have that $SS (\scrB)$ is nilpotent.
		It is clear that $p^\circ$ takes $[\fr{g}, \fr{g}] \times \scrN_{\fr{g}}$ to a subset contained in $\fr{g} \times \scrN_{\fr{g}}$ (since the nilpotent cone for a reductive Lie algebra is a subset of its semisimple part).
		Therefore, $SS (\scrC)$ is nilpotent.
	\end{proof}

	\begin{midsecproof}{Proposition \ref{thm:sing_supp_and_cs}}

	Let $\scrA$ be a character sheaf. By \ref{lem:cs_are_qad}, $\scrA$ is quasi-admissible.
	By Lemma \ref{lem:subobjs_suffices}, $\scrA$ is a subobject of $\Ind_{L \subset P}^G \scrC$ for some cuspidal sheaf $\scrC$ on $\fr{l}$.
	It then suffices to show that $SS (\Av_{P*}^G \circ i_* \pi^! \scrC)$ has nilpotent singular support. By functoriality of singular support,
	\[SS (i_* \pi^! \scrC) = i_\circ \pi^\circ SS (\scrC)\]
	which preserves nilpotency.

	By Lemma 3.3 of \cite{P}, one has that 
	\[SS (\Av_{P*}^G i_*\pi^! \scrC) \subseteq \overline{G \cdot SS(i_*\pi^! \scrC)}.\]
	Since $\scrN_{\fr{g}}$ is a union of $G$-orbits, we have that $SS (\Ind_P^G \scrC)$ is nilpotent.
	\end{midsecproof}

	\section{Induction Series}

	Cuspidal and admissible sheaves on the nilpotent cone $\scrN$ are well studied in the context of the modular generalized Springer correspondence of \cite{AHJR1}, \cite{AHJR2}, and \cite{AHJR3}.
	In this section, we will describe some analogous results for cuspidal and admissible sheaves on all of $\fr{g}$. We will also study the compatibility between admissible sheaves on $\fr{g}$ and on $\scrN$.

	Let $a_{\fr{g}} : \scrN_{\fr{g}} \to \fr{g}$ be the inclusion of the nilpotent cone into $\fr{g}$. 
	Consider the diagram

	{ \centering
	\begin{tikzcd}
		\fr{g}                                 & \fr{p} \arrow[l, "i"'] \arrow[r, "\pi"]                                   & \fr{l}                                 \\
		\scrN_{\fr{g}} \arrow[u, "a_{\fr{g}}"] & \scrN_{\fr{p}} \arrow[u, "a_{\fr{p}}"] \arrow[l, "i'"] \arrow[r, "\pi'"'] & \scrN_{\fr{l}} \arrow[u, "a_{\fr{l}}"]
		\end{tikzcd}
	\par }

	The top row of this diagram was used in defining induction and restriction functors. We will define analogous induction and restriction functors on the nilpotent cone that are compatible with restriction.
	Namely, if $A' \in \Perv_G (\scrN, \E)$, we can define the induction and restriction functors via
	\[ \nLRes_P^G = \For_L^P \circ \pi'_! \circ i'^* \circ \For_P^G \]
	\[ \nRRes_P^G = \For_L^P \circ \pi'_* \circ i'^! \circ \For_P^G \]
	\[ \nInd_P^G = \Av_{P!}^G \circ i'_! \circ \pi'^* \circ \Av_{L!}^P  \]
	These induction and restriction functors behave largely in the same way as the induction and restriction functors discussed before.
	Namely, they satisfy variants of Lemmas \ref{ind_res_verdier_duality_compat}, \ref{ind_res_fourier_compat}, \ref{ind_res_transitivity}. The following lemma explains how induction and restriction on the nilpotent cone can be computed via induction and restriction of all of $\fr{g}$.
	
	\begin{lemma}\label{ind_res_and_nilp_cone}
		Let $A \in \Perv_G (\fr{g}, \E)$ and $B \in \Perv_L (\fr{l}, \E)$. The following compatibilities between the nilpotent restriction (resp. induction) and restriction (resp. induction) functors hold 
		\begin{enumerate}
			\item $\nLRes_P^G \circ a_{\fr{g}}^* \cong a_{\fr{l}}^* \LRes_P^G$,
			\item $\LRes_P^G \circ a_{\fr{g}!} \cong a_{\fr{l}!} \nLRes_P^G$,
			\item $\nInd_P^G \circ a_{\fr{l}}^* \cong a_{\fr{g}}^* \Ind_P^G$,
			\item $\Ind_P^G \circ a_{\fr{l}!} \cong a_{\fr{g}!} \nInd_P^G$.
		\end{enumerate}
	\end{lemma}
	\begin{proof}
		Facts (2) and (4) can be found in \cite{AHJR1} 2.9. Facts (1) and (3) are an easy consequence of commutativity of the earlier diagram and proper base change.
	\end{proof}

	We recall here the notion of a cuspidal sheaf as used in the modular generalized Springer correspondence.

	\begin{definition}\label{def:nilp_cuspidal}
		Let $\scrC \in \Perv_G (\scrN, \E)$. We call $\scrC$ \textbf{cuspidal} if $\nLRes_P^G \scrC = 0$ for all proper parabolic subgroups $P \subset G$.
	\end{definition}

	\begin{remark}\label{rem:cusp_nilp_and_cusp_global}
		Observe the absence of a central factor condition in the definition of cuspidal sheaves on the nilpotent cone. This is not surprising as $\scrN_{\fr{g}} = \scrN_{[\fr{g}, \fr{g}]}$.
		In particular, if $\scrC \in \Perv_G (\scrN, \E)$ is cuspidal, then by Lemma \ref{ind_res_and_nilp_cone} (2), $a_{\fr{g}!} \scrC$ is res-small in $\fr{g}$.
		Clearly $a_{\fr{g}!} \scrC$ is supported on $[\fr{g}, \fr{g}]$, and hence is cuspidal. On the other hand, if $\scrC' \in \Perv_G (\fr{g}, \E)$ is cuspidal, then by Lemma \ref{ind_res_and_nilp_cone} (1), $\scrC\vert_{\scrN} [-\dim Z (\fr{g})]$ is cuspidal.
	\end{remark}

	The structure of the modular generalized Springer correspondence ultimately is determined by understanding the induction series. Before reviewing the modular generalized Springer correspondence, we will review the definition of induction series on the nilpotent cone as well as the generalization for the Lie algebra setting.

	\begin{definition}
		A \textbf{cuspidal datum} for $\scrN_{\fr{g}}$ is a triple $(L, \scrO, \scrE)$ where $L \subseteq G$ is a Levi subgroup, $\scrO$ is a distinguished $L$-orbit in $\scrN_{\fr{l}}$, and $\scrE \in \Loc_L (\scrO)$ such that $\IC (\scrO, \scrE)$ is a cuspidal sheaf on $\fr{l}$.
		Let $\fr{C}_{\scrN_{\fr{g}}, \E}$ be the set of all cuspidal datums on $\scrN_{\fr{g}}$.
		We can define the \textbf{induction series} $\Adm_{(L, \scrO, \scrE)} (\scrN_{\fr{g}}, \E)$ for a cuspidal datum $(L, \scrO, \scrE)$ as the full Serre subcategory of $\Perv_G (\scrN_{\fr{g}}, \E)$ generated by subobjects (or quotients) of $\nInd_L^G \IC (\scrO, \scrE)$.
	\end{definition}

	There is a $G$-action on $\fr{C}_{\scrN_{\fr{g}}, \E}$ given by
	\[g \cdot (L, \scrO, \scrE) = ({}^g L, {}^g \scrO, \Ad (g^{-1})^* \scrE )\]
	Note that the induction series is invariant under this $G$-action (Corollary 2.2, \cite{AHJR2}).

	\begin{definition}
		A \textbf{cuspidal datum} for $\fr{g}$ is a triple $(L, \alpha, \scrO, \scrE)$ where $L \subseteq G$ is a Levi subgroup, $\scrO \subseteq [\fr{l}, \fr{l}]$ is a nilpotent $L$-orbit of $\fr{l}$, $\alpha \in Z(\fr{l})$, and $\scrE \in \Loc_L (\scrO)$ such that $\scrL_\alpha \boxtimes \IC (\scrO, \scrE)$ is a cuspidal sheaf on $\fr{l}$ (where $\scrL_\alpha = \T_{Z(\fr{l})} i_{\alpha *} \underline{\E}$ ).
		Let $\fr{C}_{\fr{g}, \E}$ be the set of all cuspidal datums on $\fr{g}$.
		We can define the \textbf{induction series} $\Adm_{(L, \alpha, \scrO, \scrE)} (\fr{g}, \E)$ for a cuspidal datum $(L, \alpha, \scrO, \scrE)$ as the full Serre subcategory of $\Perv_G (\fr{g}, \E)$ generated by subobjects (or quotients) of $\Ind_L^G (\scrL_\alpha \boxtimes \IC (\scrO, \scrE))$.
	\end{definition}

	Note that by Lemma \ref{cusp_and_support_equiv} it is clear that every cuspidal sheaf on $\fr{l}$ occurs in some cuspidal datum corresponding to $L$.

	There is a $G$-action on $\fr{C}_{\fr{g}, \E}$ given by
	\[g \cdot (L, \alpha,  \scrO, \scrE) = ({}^g L, {}^g \alpha,  {}^g \scrO, \Ad (g^{-1})^* \scrE )\]
	Note that the induction series is invariant under this $G$-action by the same argument of Corollary 2.2 in \cite{AHJR2}.
	It is clear that every irreducible sheaf in an induction series is admissible. 

	Note that in both of the previous two definitions, there is seemingly a choice between using subobjects or quotients. By Lemma 2.3 of \cite{AHJR2} for the nilpotent case and Proposition \ref{prop:subobjs_equals_subquots} for the Lie algebra case, we can freely interchange these notions.

	\begin{remark}
		Let $(L, \alpha, \scrO, \scrE)$ be a cuspidal datum for $\fr{g}$, and consider the cuspidal sheaf $\scrC = \scrL_\alpha \boxtimes \IC( \scrO, \scrE)$ on $\fr{l}$.
		Note that by Remark \ref{rem:cusp_nilp_and_cusp_global}, $\IC (\scrO, \scrE)$ is cuspidal on $N_{\fr{l}}$. In particular, we can associate to this cuspidal datum for $\fr{g}$ a corresponding cuspidal datum for $\scrN_{\fr{g}}$ given by $(L, \scrO, \scrE)$.
	\end{remark}

	The modular generalized Springer correspondence can be thought of as a 3-step process of increasingly more difficult statements:
	\begin{enumerate}
		\item[(S1)] There is a decomposition of the simple perverse sheaves in $\Perv_G (\scrN, \E)$ into induction series:
		\[\Irr (\Perv_G (\scrN, \E)) = \bigcup_{(L, \scrO, \scrE) \in \fr{C}_{\scrN, \E}/G } \Irr (\Adm_{(L, \scrO, \scrE)} (\scrN, \E)).\]
		\item[(S2)] The union in (S1) is, in fact, disjoint.
		\item[(S3)] There is a bijection between the irreducible sheaves in the induction series and irreducible representations of the Weyl group of $L$:
		\[\Irr (\Adm_{L, \scrO, \scrE} (\scrN, \E)) \cong \Irr (\E [N_G (L)/L]).\]
	\end{enumerate}

	We can rephrase (S1) into the language of admissible complexes as follows: every simple perverse sheaf in $\Perv_G (\scrN, \E)$ is an admissible complex.
	This statement fails horrifically when directly translated to the entire Lie algebra. For example, consider the skyscraper sheaf $\delta_0$ at $0$ in $\Perv_G (\fr{g}, \E)$. Its Fourier transform is the constant sheaf $\underline{\E}_{\fr{g}'}$ which is not orbital; therefore, $\delta_0$ is not an admissible complex.
	Nonetheless, if one instead replaces all perverse sheaves with admissible complexes, the result does hold (Lemma \ref{lem:constr_of_cuspidal} and Lemma \ref{lem:subobjs_suffices}).

	After the refining of (S1), we can show that (S2) will hold as well. Both of these results are stated in Theorem \ref{decomp_of_ad_sheaves}.
	The proof of the Lie algebra version of (S2) will follow from a similar argument as the nilpotent cone case covered in \cite{AHJR3}. Namely, one must leverage a version of Mackey's formula in positive characteristic.
	We provide a version of Mackey's formula in Theorem \ref{mackeys_thm}.  

	There likely should be a variant for (S3) as well, although we do not yet have a sensible replacement. 
	If one looks at all of $\Perv_G (\fr{g}, \E)$ with $\textnormal{char} (\E) = 0$, there is a replacement given in \cite{G18}, but it is not clear how to adapt these methods in our case.

	\subsection{Mackey's Formula}

	We will establish a variation of Mackey's formula which holds for character sheaves on $\fr{g}$. The proof will follow very closely from Theorem 2.2 of \cite{AHJR3} and Proposition 10.1.2 of \cite{MS}.

	Let $G$ be an algebraic group with normal subgroup $H \trianglelefteq G$. Let $X$ be a $G$-variety such that $X$ is a principal $H$-variety. The quotient equivalence states that for $\pi : X \to X/H$, $\pi^* \circ \textnormal{Infl}_{G/H}^G: D_{G/H} (X/H, \E) \to D_G (X, \E)$ is an equivalence of categories.
	We will write $\pi_\flat$ for the inverse of this equivalence. The following lemma gives a reformulation of the parabolic induction functor. The new formulation closely models the definition of parabolic induction given in Lusztig \cite{CS1}.

	\begin{lemma}[\cite{AHJR1}, Lemma 2.14]\label{reformulation_of_ind}
		Consider the diagram
		\begin{center}
			\begin{tikzcd}
				\fr{l} & G \times \fr{p} \arrow[l, "\alpha"'] \arrow[r, "\beta"] & G \times^P \fr{p} \arrow[r, "\gamma"] & \fr{g}
				\end{tikzcd}	
		\end{center}
		where $\alpha (g,x) = \pi_P (x)$ is the projection of $x$ onto $\fr{l}$, $\beta$ is the obvious quotient, and $\gamma$ is the map such that $\gamma \circ \beta = \sigma$ is the action map $\sigma (g,x) = {}^g x$.
		Then there is a natural isomorphism of functors,
		\[\Ind_P^G \cong \gamma_! \beta_\flat \alpha^* [2 \dim U_P]\]
		where $U_P$ is the unipotent radical of $P$.
	\end{lemma}

	We will first state a result about double cosets. We will use the notation $^g L = gLg^{-1}$.
	\begin{lemma}[\cite{DM}, Lemma 5.6(i)]\label{ordering_levis}
		Let $P$, $Q$ be parabolic subgroups of $G$ with Levi factors $L$, $M$ respectively. Define
		\[ \Sigma_{M,L} = \{g \in G \mid M \cap {}^g L \textnormal{ contains a maximal torus of } G \}.\]
		Then $\Sigma_{M,L}$ is a union of finitely many $M\text{-}L$ double cosets and there is a bijection
		\[ M \setminus \Sigma_{M, L} / L \leftrightarrow Q \setminus G / P.\]
	\end{lemma}

	Given parabolics $P,Q$ with Levis $L, M$, let $g_1, g_2, \ldots, g_s$ be a set of representatives for the $M\text{-}L$ double cosets in $\Sigma_{M,L}$ ordered such that $Q g_i P \subseteq \overline{Q g_j P}$ when $i \leq j$. Note that for $g_i \in \Sigma_{M, L}$, the group $M \cap {}^{g_i} L$ is a Levi subgroup of both $M$ and ${}^{g_i} L$. 
	Precisely, it is a Levi factor of the parabolic subgroup of $M \cap {}^{g_i} P$ of $M$ and a Levi factor of the parabolic subgroup $Q \cap {}^{g_i} L$ of ${}^{g_i} L$.

	For a parabolic subgroup $P$ with Levi factor $L$, denote the projections $\pi_{L \subset P} : P \to L$ and $\pi_{\fr{l} \subset \fr{p}} : \fr{p} \to \fr{l}$.

	We will state Mackey's formula in the modular setting. The statement is similar to the characteristic 0 case, but we cannot obtain a direct sum due to the failure of the theory of weights.
	\begin{theorem}\label{mackeys_thm}
		Let $L \subset P$ be a parabolic subgroup with Levi subgroup and $M \subset Q$ be another parabolic subgroup with Levi subgroup. 
		The corresponding Lie algebras will be $\fr{l}, \fr{p}, \fr{m},$ and $\fr{q}$ respectively.
		Let $g_1, g_2, \ldots, g_s$ be as above. Let $\scrA \in \Perv_L (\fr{l}, \E)$. Then in $\Perv_M (\fr{m}, \E)$, we have a filtration
		\[ \LRes_{M \subset Q}^G \Ind_{L \subset P}^G \scrA = \scrA_0 \supset \scrA_1 \supset A_2 \supset \ldots \supset A_s = 0, \]
		in which the successive quotients have the form
		\[ \scrA_{i-1}/\scrA_i \cong \Ind_{M \cap^{g_i} L \subset M \cap^{g_i} P}^M \LRes_{M \cap^{g_i} L \subset Q \cap^{g_i} L}^{{}^{g_i} L} \left( \Ad (g_i^{-1})^* \scrA \right), \:\:\:\: \text{for } i=1, \ldots, s.\]
	\end{theorem}
	\begin{proof}
		The role of equivariance in this proof is often subtle. To have this proof mimic the one of Mars and Springer, we will frequently regard $\Perv_L (\fr{l}, \E)$ as a subcategory of $D (\fr{l}, \E)$.
		Throughout the proof, we will use the reformulation of parabolic induction, as in Lemma \ref{reformulation_of_ind}.

		Define a variety
		\[ Z = \{ (g,x) \in G \times \fr{p} \mid {}^g x \in \fr{q} \} \]
		We will define maps analogous to $\alpha, \beta, \gamma$ on $Z$. First, define $a : Z \to \fr{l}$ by $a (g,x) = \pi_{\fr{l} \subset \fr{p}} (x)$ and $\tau : Z \to \fr{m}$ by $\tau (g,x) = \pi_{\fr{m} \subset \fr{q}} ({}^g x)$.
		Just as with $\sigma$, we can factor $\tau$ into maps
		\begin{center}
		\begin{tikzcd}
			Z \arrow[rd, "b"] \arrow[rr, "\tau"] &                                                & \fr{m}                                 \\
											  & Z' \arrow[ru, "\phi"] \arrow[r, "d"] & \fr{q} \arrow[u, "\pi_{\fr{q} \subset \fr{m}}"']
			\end{tikzcd}
		\end{center}
		where $Z' = \{ (g,x) \in G \times^P \fr{p} \mid {}^g x \in \fr{q} \}$, and $d$ is the obvious restriction of $\gamma$ onto $Z'$. This gives the following commutative diagram

		{ \centering
		\begin{tikzcd}
			&                                                                      &                                                              & \fr{m}                                                             \\
			& Z \arrow[r, "b"'] \arrow[rru, "\tau"] \arrow[ld, "a"'] \arrow[d, "i_Z"] & Z' \arrow[ru, "\phi"] \arrow[r, "d"'] \arrow[d, "i_{Z'}"] & \fr{q} \arrow[u, "\pi_{\fr{q} \subset \fr{m}}"'] \arrow[d, "i_{\fr{q} \subset \fr{g}}"] \\
	 \fr{l} & G \times \fr{p} \arrow[l, "\alpha"'] \arrow[r, "\beta"]              & G \times^P \fr{p} \arrow[r, "\gamma"]                        & \fr{g}                                                            
	 \end{tikzcd}
		\par }

		We can now reduce computing the restriction of induction to computing sheaf functors on $Z$ and its related varieties. In particular, we can compute
		\begin{align}
			\LRes_{\fr{q}}^{\fr{g}} \Ind_P^G \scrA &= (\pi_{\fr{q} \subset \fr{m}})_! i_{\fr{q} \subset \fr{g}} ^* \delta_! \beta_\flat \alpha^* \scrA [2\dim U_P] \notag \\
			&= (\pi_{\fr{q} \subset \fr{m}})_! d_! i_{Z'}^* \beta_\flat \alpha^* \scrA [2\dim U_P] \notag \\
			&= \phi_! i_{Z'}^* \beta_\flat \alpha^* \scrA [2\dim U_P] \notag \\
			&= \phi_! b_\flat i_Z^* \alpha^* \scrA [2\dim U_P] \notag \\
			&= \phi_! b_\flat a^* \scrA [2\dim U_P] \notag
		\end{align}
		
		As it stands, the variety $Z$ needs to be broken up so we can inductively prove the theorem on the length of the filtration. 
		We will first define closed subvarieties of $G$,
		\[ G_i = \bigcup_{j=1}^i Q g_j P \]
		for $i = 0, \ldots, s$. Note that $G_0 = \emptyset$ and $G_s = G$. We will now use these to restrict $Z$. Define
		\[ Z_i = \{ (g,x) \in Z \mid g \in G_i \} = \{ (g,x) \in G_i \times \fr{p} \mid {}^g x \in \fr{p} \} \]
		Similarly, $Z_0 = \emptyset$ and $Z_s = Z$. To obtain the piece of $Z$ only consisting of the $g_i$ coset, we can define
		\[ V_i = Z_i \setminus Z_{i-1}. \]
		We can create subsets $Z_i'$ and $V_i'$ in the same way that $Z'$ was defined. We similarly get factorization maps. We can encode all this information in the following commutative diagrams

		{ \centering
		\begin{tikzcd}
			\fr{l} \arrow[d, "\id"] & Z_1 \arrow[l, "a_1"] \arrow[r, "b_1"'] \arrow[rr, "\tau_1", bend left] \arrow[d, "\iota_1"]                     & Z_1' \arrow[r, "\phi_1"'] \arrow[d, "\iota_1'"]            & \fr{m} \arrow[d, "\id"] \\
			\fr{l} \arrow[d, "\id"] & Z_2 \arrow[l, "a_2"] \arrow[r, "b_2"'] \arrow[rr, "\tau_2", bend left] \arrow[d, "\iota_2"]                     & Z_2' \arrow[r, "\phi_2"'] \arrow[d, "\iota_2'"]            & \fr{m} \arrow[d, "\id"] \\
			\fr{l} \arrow[d, "\id"] & \vdots \arrow[l, "a_i"] \arrow[r, "b_i"'] \arrow[rr, "\tau_i", bend left] \arrow[d, "\iota_i"]                  & \vdots \arrow[r, "\phi_i"'] \arrow[d, "\iota_i'"]          & \fr{m} \arrow[d, "\id"] \\
			\fr{l} \arrow[d, "\id"] & Z_{s-1} \arrow[l, "a_{s-1}"] \arrow[r, "b_{s-1}"'] \arrow[rr, "\tau_{s_1}", bend left] \arrow[d, "\iota_{s-1}"] & Z_{s-1}' \arrow[r, "\phi_{s-1}"'] \arrow[d, "\iota_{s-1}"] & \fr{m} \arrow[d, "\id"] \\
			\fr{l}                  & Z \arrow[l, "a"] \arrow[r, "b"'] \arrow[rr, "\tau", bend left]                                              & Z' \arrow[r, "\phi"']                                  & \fr{m}                 
			\end{tikzcd}
		\par }

		and

		{ \centering 
		\begin{tikzcd}
			\fr{l} & V_i \arrow[d, "j"] \arrow[r, "\scrb_i"] \arrow[l, "\scra_i"'] & V_i' \arrow[d, "j'"] \arrow[r, "\varphi_i"] & \fr{m} \\
				   & Z_i \arrow[lu, "a_i"] \arrow[r, "b_i"']                   & Z_i' \arrow[ru, "\phi_i"']                  &       
			\end{tikzcd}
		\par }

		To make the notation a little cleaner for the next part, let $\iota : Z_{i-1} \to Z_i$ be the closed embedding normally denoted $\iota_{i-1}$ and $j : V_i \to Z_i$ be the open embedding. 
		We also denote $\iota' = \iota_{i-1}'$.
		We then have an open-closed distinguished triangle associated with these embeddings

		{ \centering
		\begin{tikzcd}[column sep=tiny, row sep=small]
			{j_!j^* a_i^* \scrA [2\dim U_P]} \arrow[r]                              & {a_i^* \scrA [2\dim U_P]} \arrow[r]                                     & {\iota_! \iota^* a_i^* \scrA [2\dim U_P]} \arrow[r]                                      & {} \\
			{j_! \scra_i^* \scrA [2\dim U_P]} \arrow[r]                             & {a_i^* \scrA [2\dim U_P]} \arrow[r]                                     & {\iota_! a_{i-1}^* \scrA [2\dim U_P]} \arrow[r]                                     & {} \\
			{(b_i)_\flat j_! \scra_i^* \scrA [2\dim U_P]} \arrow[r]                 & {(b_i)_\flat a_i^* \scrA [2\dim U_P]} \arrow[r]                         & {(b_i)_\flat \iota_! a_{i-1}^* \scrA [2\dim U_P]} \arrow[r]                         & {} \\
			{j_!' (\scrb_i)_\flat \scra_i^* \scrA [2\dim U_P]} \arrow[r]            & {(b_i)_\flat a_i^* \scrA [2\dim U_P]} \arrow[r]                         & {\iota'_! (b_{i-1})_\flat a_{i-1}^* \scrA [2\dim U_P]} \arrow[r]                    & {} \\
			{(\phi_i)_! j_!' (\scrb_i)_\flat \scra_i^* \scrA [2\dim U_P]} \arrow[r] & {(\phi_i)_!  (b_i)_\flat a_i^* \scrA [2\dim U_P]} \arrow[r] & {(\phi_i)_! \iota_!' (b_{i-1})_\flat a_{i-1}^* \scrA [2\dim U_P]} \arrow[r] & {} \\
			{(\varphi_i)_! (\scrb_i)_\flat \scra_i^* \scrA [2\dim U_P]} \arrow[r]     & {(\phi_i)_!  (b_i)_\flat a_i^* \scrA [2\dim U_P]} \arrow[r] & {(\phi_{i-1})_! (b_{i-1})_\flat a_{i-1}^* \scrA [2\dim U_P]} \arrow[r]  & {}
			\end{tikzcd}
		\par }

		The following lemma will be crucial to the proof and can be thought of as the inductive step of our proof. We will delay the proof until later.

		\begin{lemma}\label{mackeys_inductive_step}
			In the same setup as the proof of the theorem,
			\[ (\varphi_i)_! (\scrb_i)_\flat \scra_i^* \scrA [2\dim U_P] \cong \Ind_{M \cap^{g_i} L \subset M \cap^{g_i} P}^M \LRes_{M \cap^{g_i} L \subset Q \cap^{g_i} L}^{{}^{g_i} L} (\Ad (g_i^{-1})^* \scrA)\]
		\end{lemma}

		We will first perform an induction to show that $\scrC_i = (\phi_i)_! (b_i)_\flat a_i^* \scrA [2\dim U_P]$ is perverse for all $i$.
		We will also denote $\scrB_i = (\varphi_i)_! (\scrb_i)_\flat \scra^* \scrA [2\dim U_P]$. By Lemma \ref{mackeys_inductive_step}, $\scrB_i$ is perverse for all $i$.
		The base case is easy as the distinguished triangle becomes the exact sequence,
		\[ 0 \to \scrB_1 \stackrel{\sim}{\to} \scrC_1 \to 0\]
		and hence $\scrC_1$ is perverse. Now assume that $\scrC_{i-1}$ is perverse. Then the distinguished triangle is
		\[ \scrB_i \to \scrC_i \to \scrC_{i-1} \to \]
		Since perverse sheaves are closed under extensions, we have that $\scrC_i$ is perverse. In particular, the distinguished triangle becomes a short exact sequence of perverse sheaves
		\[ 0 \to \scrB_i \to \scrC_i \to \scrC_{i-1} \to 0. \label{eq:mackeys_inductive_step_1}\]

		It remains to prove the actual substance of the theorem. We will do this by inducting on the length of the filtration. For the base case of $i=1$, the theorem obviously holds since $\scrC_1 \cong \scrB_1$ by the above short exact sequence.
		Now assume that $\scrC_{i-1}$ has a descending filtration of length $i-1$ with the same successive quotients as the theorem. By (\ref{eq:mackeys_inductive_step_1}), we can extend this to $\scrC$.
		Hence by induction, the result holds for all $i$. We can then take $i=s$ and then we are done by Lemma \ref{mackeys_inductive_step}.
	\end{proof}

	\begin{midsecproof}{Lemma \ref{mackeys_inductive_step}}
		Fix $i = 1, \ldots ,s$. We can perform a short computation for $V_i$,
		\begin{align*}
			V_i &= \{ (g,x) \in G_i \setminus G_{i-1} \times \fr{p} \mid {}^g x \in \fr{q} \} \\
			&= \left\{ (g,x) \in \left( \bigcup_{j=1}^i Qg_j P \setminus \bigcup_{j=1}^{i-1} Q g_j P \right) \times \fr{p} \: \bigg\vert \: {}^g x \in \fr{q} \right\} \\
			&= \{ (g,x) \in Qg_i P \mid {}^g x \in \fr{q} \}.
		\end{align*}
		Define $Y_i = Q \times P \times (\fr{q} \cap {}^{g_i} \fr{p})$ which is a $Q \cap {}^{g_i} P$-variety via the action,
		\[g \cdot (q,p,y) = (qg^{-1}, pg_i^{-1} g^{-1} g_i, {}^g y)\]
		for $g \in Q \cap {}^{g_i} P$, $(q,p,y) \in Y_i$. In fact, $Y_i$ is a principal $Q \cap {}^{g_i} P$-bundle over $V_i$ via the map
		\[\sigma : Y_i \to V_i \hspace{1cm} (q,p,y) \mapsto (qg_i p^{-1}, {}^{pg_i^{-1}} y).\]

		We can define variations of $\varphi_i$, $\scrb_i$ and $\scra_i$ for $Y_i$. Namely, define maps
		\[ \tilde{a} : Y_i \to \fr{l} \hspace{1cm} (q,p,y) \mapsto \pi_{\fr{l} \subset \fr{p}} ({}^{pg_i^{-1}} y),\]
		\[ \tilde{\tau} : Y_i \to \fr{m} \hspace{1cm} (q,p,y) \mapsto \pi_{\fr{m} \subset \fr{q}} ({}^q y).\]
		We can factor $\tilde{\tau} = \tilde{\varphi} \circ \tilde{b}$ where $\tilde{b}: Y_i \to Q \times^{Q \cap {}^{g_i} P} (\fr{q} \cap {}^{g_i} \fr{p})$ is the projection map.
		We summarize all these maps in the following commutative diagram,

		{ \centering
		\begin{tikzcd}
			\fr{l} \arrow[d, "\id"'] & Y_i \arrow[l, "\tilde{a}"'] \arrow[rr, "\tilde{\tau}", bend left=24] \arrow[r, "\tilde{b}"'] \arrow[d, "\sigma"'] & Q \times^{Q \cap {}^{g_i} P} (\fr{q} \cap {}^{g_i} \fr{p}) \arrow[r, "\tilde{\varphi}"'] \arrow[d, "\sigma'"'] & \fr{m} \arrow[d, "\id"'] \\
			\fr{l}                   & V_i \arrow[l, "\scra_i"] \arrow[r, "\scrb_i"']                                                                    & V_i' \arrow[r, "\varphi_i"']                                                                                & \fr{m}                  
			\end{tikzcd}
		\par }

		As we have done before, we can find an isomorphism of sheaves,
		\begin{equation}
			(\varphi_i)_! (\scrb_i)_\flat \scra_i^* \scrA [2\dim U_P] \cong \tilde{\varphi}_! \tilde{b}_\flat \tilde{a}^* \scrA [2\dim U_P] \label{eq:mackey_ind_1}
		\end{equation}

		Define a morphism
		\[q_L : Y_i \to L \times \fr{l} \hspace{1cm} (q,p,y) \mapsto (\pi_{L \subset P} (p), \pi_{\fr{l} \subset \fr{p}} ({}^{g_i^{-1}} y) )\]
		It is easy to see that $\tilde{\alpha}$ factors as $a_{\fr{l}} \circ q_L$ where $a_{\fr{l}}$ is the action map. We can then consider the diagram,
		\begin{center}
		\begin{tikzcd}
			Y_i \arrow[r, "q_L"] \arrow[rr, "\tilde{a}", bend left=49] \arrow[rr, "\tilde{a}'"', bend right=49] & L \times \fr{l} \arrow[r, "a_{\fr{l}}", shift left] \arrow[r, "\pr_2"', shift right] & \fr{l}
			\end{tikzcd}
		\end{center}
		where $\tilde{a}'$ is the composition $\pr_2 \circ q_L$. It is clear that both the top and bottom triangles of the diagram commute.
		Since $\scrA$ is $L$-equivariant, we have an isomorphism $a_{\fr{l}}^* \scrA \cong \pr_2^* \scrA$. Therefore, we get an isomorphism $\tilde{a}^* \scrA \cong (\tilde{a}')^* \scrA$.
		The main benefit now is that $\tilde{a}'$ is independent of the $P$-factor. Therefore, we can factor $\tilde{a}'$ through the quotient $Y_i \to Q \times (\fr{q} \cap {}^{g_i} \fr{p})$ to get a map $\hat{a} : Q \times (\fr{q} \cap {}^{g_i} \fr{p}) \to \fr{l}$. 
		Likewise, we can factor $\tilde{\tau}$ through the quotient $Y_i \to Q \times (\fr{q} \cap {}^{g_i} \fr{p})$ to get a map $\hat{\tau} : Q \times (\fr{q} \cap {}^{g_i} \fr{p}) \to \fr{m}$.
		We can further factor $\hat{\tau}$ as $\tilde{\varphi}\hat{b}$ where $\hat{b} : Q \times (\fr{q} \cap {}^{g_i} \fr{p}) \to Q \times^{Q \cap {}^{g_i} P} (\fr{q} \cap {}^{g_i} \fr{p})$ is the usual quotient map.
		All of these maps can be encoded in the commutative diagram:

		{ \centering
		\begin{tikzcd}
			\fr{l} \arrow[d, "\id"] & Y_i \arrow[r, "\tilde{b}"'] \arrow[l, "\tilde{a}"'] \arrow[d] \arrow[rr, "\tilde{\tau}", bend left]                     & Q \times^{Q \cap {}^{g_i} P} (\fr{q} \cap {}^{g_i} \fr{p}) \arrow[r, "\tilde{\varphi}"'] \arrow[d, "\id"] & \fr{m} \arrow[d, "\id"] \\
			\fr{l}                  & Q \times (\fr{q} \cap {}^{g_i} \fr{p}) \arrow[r, "\hat{b}"] \arrow[l, "\hat{a}"'] \arrow[rr, "\hat{\tau}"', bend right] & Q \times^{Q \cap {}^{g_i} P} (\fr{q} \cap {}^{g_i} \fr{p}) \arrow[r, "\tilde{\varphi}"]                   & \fr{m}                 
			\end{tikzcd}
		\par }

		By basic sheaf functor relations, we get an isomorphism of sheaves,
		\begin{equation}
			\tilde{\varphi}_! \tilde{b}_\flat \tilde{a}^* \scrA \cong \tilde{\varphi}_! \hat{b}_\flat \hat{a}^* \scrA  \label{eq:mackey_ind_2}
		\end{equation}

		Define 
		\[ Y_i' = (\fr{m} \cap {}^{g^i} \fr{l}) \oplus (\fr{m} \cap {}^{g^i} \fr{u}_P) \oplus (\fr{u}_Q \cap {}^{g^i} \fr{l}) \]
		which comes equipped with an obvious projection $\rho : \fr{q} \cap {}^{g_i} \fr{p} \to Y_i'$. This projection map is $(Q \cap {}^{g_i} P)$-equivariant.
		We also get an induced vector bundle projection $\id \times \rho: Q \times (\fr{q} \cap {}^{g_i} \fr{p}) \to Q \times Y_i'$.
		Factor $\hat{a}$ and $\hat{\tau}$ through $(\id \times \rho)$ to get maps $\hat{a} = \check{a} (\id \times \rho)$ and $\hat{\tau} = \check{\tau} (\id \times \rho)$.
		We can further factor $\check{\tau}$ as $\check{\varphi} \check{b}$ where $\check{b} : Q \times Y_i' \to Q \times^{Q \cap {}^{g_i} P} Y_i'$ is the obvious projection. All of these maps can be encoded in the commutative diagram:

		{ \centering
		\begin{tikzcd}
			\fr{l} \arrow[d, "\id"] & Q \times (\fr{q} \cap {}^{g_i} \fr{p}) \arrow[r, "\hat{b}"'] \arrow[l, "\hat{a}"'] \arrow[rr, "\hat{\tau}", bend left] \arrow[d, "\id \times \rho"] & Q \times^{Q \cap {}^{g_i} P} (\fr{q} \cap {}^{g_i} \fr{p}) \arrow[r, "\tilde{\varphi}"'] \arrow[d] & \fr{m} \arrow[d, "\id"] \\
			\fr{l}                  & Q \times Y_i' \arrow[r, "\check{b}"] \arrow[rr, "\check{\tau}"', bend right] \arrow[l, "\check{a}"']                                                & Q \times^{Q \cap {}^{g_i} P} Y_i' \arrow[r, "\check{\varphi}"]                                     & \fr{m}                 
			\end{tikzcd}
		\par }

		We then obtain an isomorphism of sheaves,
		\begin{equation}
			\tilde{\varphi}_! \hat{b}_\flat \hat{a}^* \scrA \cong \check{\varphi}_! \check{b}_\flat \check{a}^* \scrA [-2\dim (U_Q \cap {}^{g_i} U_P)]. \label{eq:mackey_ind_3}
		\end{equation}

		We could have instead factored $\check{\tau}$ as $\check{\varphi}' \check{b}'$ where $\check{b}' : Q \times Y_i' \to Q \times^{M \cap {}^{g_i} P} Y_i'$. 
		Note that $U_Q \cap {}^{g_i} P$ is affine, so we get an isomorphism
		\begin{equation}
			\check{\varphi}_! \check{b}_\flat \check{a}^* \scrA \cong \check{\varphi}_!' \check{b}_\flat' \check{a}^* \scrA [2\dim (U_Q \cap {}^{g_i} U_P)]. \label{eq:mackey_ind_4}
		\end{equation}

		Consider the affine bundle $p_{M \subset Q} \times \id : Q \times Y_i' \to M \times Y_i'$. It is not hard to check that both $\check{a}$ and $\check{\tau}$ both factor through $p_{M \subset Q} \times \id$.
		We then get factorizations $\check{a} = \overline{a} (p_{M \subset Q} \times \id)$ and $\check{\tau} = \overline{\tau} (p_{M \subset Q} \times \id)$.
		Again, we can further factor $\overline{\tau}$ as $\overline{\tau} = \overline{\varphi} \overline{b}$ where $\overline{b} : M \times Y_i' \to M \times^{M \cap {}^{g_i} P} Y_i'$ is the quotient map. These maps fit into the commutative diagram,

		{ \centering
		\begin{tikzcd}
			\fr{l} \arrow[d, "\id"] & Q \times Y_i' \arrow[r, "\check{b}'"'] \arrow[rr, "\check{\tau}", bend left] \arrow[l, "\check{a}"'] \arrow[d, "p_{M \subset Q} \times \id"] & Q \times^{M \cap {}^{g_i} P} Y_i' \arrow[r, "\check{\varphi}'"'] \arrow[d] & \fr{m} \arrow[d, "\id"] \\
			\fr{l}                  & M \times Y_i' \arrow[l, "\overline{a}"] \arrow[r, "\overline{b}"] \arrow[rr, "\overline{\tau}"', bend right]                                 & M \times^{M \cap {}^{g_i} P} Y_i' \arrow[r, "\overline{\varphi}"]          & \fr{m}                 
			\end{tikzcd}
		\par }

		By another series of basic sheaf functor relations, we obtain an isomorphism,
		\begin{equation}
			\check{\varphi}_!' \check{b}_\flat' \check{a}^* \scrA \cong \overline{\varphi}_! \overline{b}_\flat \overline{a}^* \scrA [-2\dim U_Q]. \label{eq:mackey_ind_5}
		\end{equation}

		Finally, we will compare the right hand side of (\ref{eq:mackey_ind_5}) with the definition of $\Ind_{M \cap^{g_i} L \subset M \cap^{g_i} P}^M \LRes_{M \cap^{g_i} L \subset Q \cap^{g_i} L}^{{}^{g_i} L} (\Ad (g_i^{-1})^* \scrA)$. 
		Namely, we can fit all of the relevant information into the following commutative diagram,

		{ \centering
		\begin{tikzcd}[row sep=small, column sep=tiny]
			\fr{l}                   & {}^{g_i} \fr{l} \arrow[l, "\Ad (g_i^{-1})"'] & \fr{q} \cap {}^{g_i} \fr{l} \arrow[l, "i"'] \arrow[r, "\pi"] & \fr{m} \cap^{g_i} \fr{l}                              & M \times (\fr{m} \cap {}^{g_i} \fr{p}) \arrow[l, "\overline{\alpha}"'] \arrow[r, "\overline{\beta}"] \arrow[ld] & M \times^{M \cap {}^{g_i} P} (\fr{m} \cap {}^{g_i} \fr{p}) \arrow[r, "\overline{\delta}"] & \fr{m}                   \\
									 &                                              & Y_i' \arrow[r, "\pi'"] \arrow[u, "\overline{p}'"]            & \fr{m} \cap {}^{g_i} \fr{p} \arrow[u, "\overline{p}"] &                                                                                                                 &                                                                                           &                          \\
			\fr{l} \arrow[uu, "\id"] &                                              &                                                              &                                                       & M \times Y_i' \arrow[r, "\overline{b}"] \arrow[llll, "\overline{a}"] \arrow[uu, "p_1"] \arrow[llu]              & M \times^{M \cap {}^{g_i} P} Y_i' \arrow[r, "\overline{\varphi}"] \arrow[uu, "p_1'"]         & \fr{m} \arrow[uu, "\id"]
			\end{tikzcd}
		\par }

		First, note that by definition, 
		\[\Ind_{M \cap^{g_i} L \subset M \cap^{g_i} P}^M \LRes_{M \cap^{g_i} L \subset Q \cap^{g_i} L}^{{}^{g_i} L} (\Ad (g_i^{-1})^* \scrA) = \overline{\delta}_! \overline{\beta}_\flat \overline{\alpha}^* \pi_! i^* \Ad (g_i^{-1})^* \scrA \]
		By the same arguments we have performed throughout this proof, we can reduce to showing that
		\[\overline{a}^*  = p_1^* \overline{\alpha}^* \pi_! i^* \Ad (g_i^{-1})^* \]
		which follows from a combination of proper base change of the middle square and the fact that $\pi'$ is an affine bundle. Therefore, we get an isomorphism of sheaves,
		\begin{equation}
			\Ind_{M \cap^{g_i} L \subset M \cap^{g_i} P}^M \LRes_{M \cap^{g_i} L \subset Q \cap^{g_i} L}^{{}^{g_i} L} (\Ad (g_i^{-1})^* \scrA) \cong \overline{\varphi}_! \overline{b}_\flat \overline{a}^* \scrA [2\dim(M \cap {}^{g_i} P)]. \label{eq:mackey_ind_6}
		\end{equation}
		We can then combine the myriad of isomorphisms from (\ref{eq:mackey_ind_1}), (\ref{eq:mackey_ind_2}), (\ref{eq:mackey_ind_3}), (\ref{eq:mackey_ind_4}), (\ref{eq:mackey_ind_5}), (\ref{eq:mackey_ind_6}).
		It only remains to check that the shifts accumulated throughout cancel out. I.e., we must check that
		\begin{equation}
			\dim U_P - \dim (U_Q \cap {}^{g_i} U_P) + \dim (U_Q \cap {}^{g_i} P) - \dim (U_Q) = \dim (M \cap {}^{g_i} U_P). \label{eq:mackey_ind_7}
		\end{equation}
		Note that (\ref{eq:mackey_ind_7}) follows from the proof of Proposition 10.1.2 of \cite{MS}.
	\end{midsecproof}

	\subsection{Decomposition of Admissible Sheaves}

	To simplify notation, for a cuspidal datum $(L, \alpha, \scrO, \scrE)$ for $\fr{g}$, we will write $\scrI_{L, \scrO}^{\alpha, \scrE} = \scrL_\alpha \boxtimes \IC (\scrO, \scrE)$.
	
	\begin{lemma}\label{intersection_of_ind_series}
		Let $P_1, P_2$ be Levi subgroups of $G$ with Levi radicals $L_1, L_2$. Let $(L_1, \alpha_1, \scrO_1, \scrE_1)$ and $(L_2, \alpha_2, \scrO_2, \scrE_2)$ be two cuspidal datums for $\fr{g}$.
		If $\scrA \in \Perv_G (\fr{g}, \E)$ belongs to the induction series of both $(L_1, \alpha_1, \scrO_1, \scrE_1)$ and $(L_2, \alpha_2, \scrO_2, \scrE_2)$, then 
		\[ \Hom \left( \Ind_{P_1}^{G} \scrI_{L_1, \scrO_1}^{\alpha_1, \scrE_1}, \Ind_{P_2}^G \scrI_{L_2, \scrO_2}^{\alpha_2, \scrE_2} \right) \neq 0\]
	\end{lemma}
	\begin{proof}
		By definition, $\scrA$ is a subobject of $\Ind_{P_2}^G \scrI_{L_2, \scrO_2}^{\alpha_2, \scrE_2}$ and a quotient of $\Ind_{P_1}^G \scrI_{L_1, \scrO_1}^{\alpha_1, \scrE_1}$.
		We then have the composition of maps
		\[ \Ind_{P_1}^G \scrI_{L_1, \scrO_1}^{\alpha_1, \scrE_1} \twoheadrightarrow  \scrA \hookrightarrow \Ind_{P_2}^G \scrI_{L_2, \scrO_2}^{\alpha_2, \scrE_2} \]
		In particular, the composition is nonzero, so the hom-space is nonzero.
	\end{proof}

	\begin{theorem}\label{decomp_of_ad_sheaves}
		There is a decomposition of the admissible sheaves into induction series, 
		\[ \Irr (\Adm_G (\fr{g}, \E)) = \bigsqcup_{(L, \alpha, \scrO, \scrE) \in \fr{C}_{\fr{g}, \E}/G } \Irr (\Adm_{(L, \alpha, \scrO, \scrE)} (\fr{g}, \E)).\]
	\end{theorem}
	\begin{proof}
		By Lemma \ref{lem:constr_of_cuspidal} and Lemma \ref{lem:subobjs_suffices}, we have that every admissible sheaf occurs in some induction series. As a result, it only remains to show that this union is disjoint.

		Suppose $(L_1, \alpha_1, \scrO_1, \scrE_1)$ and $(L_2, \alpha_2, \scrO_2, \scrE_2)$ are cuspidal datums for $\fr{g}$ which are not in the same $G$-orbit. By adjunction, 
		\[ \Hom \left( \Ind_{P_1}^{G} \scrI_{L_1, \scrO_1}^{\alpha_1, \scrE_1}, \Ind_{P_2}^G \scrI_{L_2, \scrO_2}^{\alpha_2, \scrE_2}\right) = \Hom \left( \scrI_{L_1, \scrO_1}^{\alpha_1, \scrE_1}, \RRes_{P_1}^{G} \Ind_{P_2}^G \scrI_{L_2, \scrO_2}^{\alpha_2, \scrE_2}\right)\]
		By Mackey's theorem, we have a filtration
		\[ \RRes_{L_1 \subset P_1}^G \Ind_{L_2 \subset P_2}^G \scrI_{L_2, \scrO_2}^{\alpha_2, \scrE_2} = \scrA_0 \supset \scrA_1 \supset A_2 \supset \ldots \supset A_s = 0, \]
		in which the successive quotients have the form
		\[ \scrA_{i-1}/\scrA_i \cong \Ind_{L_1 \cap^{g} L_2 \subset L_1 \cap^{g} P_2 g^{-1}}^M \RRes_{L_1 \cap^{g} L_2 \subset P_1 \cap^{g} L_2}^{{}^{g} L_2} \left( \Ad (g^{-1})^* \scrI_{L_2, \scrO_2}^{\alpha_2, \scrE_2} \right)\]
		Since $\Ad (g^{-1})$ is a smooth morphism, $\Ad (g^{-1})^* \scrI_{L_2, \scrO_2}^{\alpha_2, \scrE_2}$ is also cuspidal. Hence, $\scrA_{i-1}/\scrA_i$ must be zero.
		Since all of the composition factors are zero, we must have that $\RRes_{L_1 \subset P_1}^G \Ind_{L_2 \subset P_2}^G \scrI_{L_2, \scrO_2}^{\alpha_2, \scrE_2}= 0$.
		But by Lemma \ref{intersection_of_ind_series}, this can only happen if the induction series are disjoint.
	\end{proof}

	\section{Modular Reduction}

	For an algebraic variety $X$, consider the extension by scalars functor $\K (-) = \K \otimes_\O ( - ): D (X, \O) \to D (X, \K)$. We can use $\K (-)$ to identify the objects of $D (X, \K)$ with the objects of $D (X, \O)$. We can compute the hom-spaces of $D (X, \K)$ from $D (X, \O)$ using the formula
	\[ \Hom_{D (X,\K)} (A,B) = \K \otimes_\O \Hom_{D (X, \O)} (A,B)\]
	under this identification of objects. Moreover, a similar relationship holds when restricting to the heart of the perverse $t$-structure since $\K (-)$ is $t$-exact.

	On the other hand, we have a restriction of scalars functor $\F (-) = \F \otimes_\O^L (-)$ induced by the quotient $\O \twoheadrightarrow \F$. 

	We have an induced map on Grothendieck groups, called the \textbf{modular reduction map} 
	\[ d : K_0 (\Perv_G (\fr{g}, \K)) \to K_0 (\Perv_G (\fr{g}, \F))\]
	For $\scrA \in \Perv_G (\fr{g}, \K)$, let $\scrA_{\O}$ be the torsion-free sheaf in $\Perv_G (\fr{g}, \O)$ such that $\scrA = \K (\scrA_{\O})$.
	Then $d([\scrA]) = [\F (\scrA_{\O})]$. We say that $\scrA \in \Perv_G (\fr{g}, \F)$ belongs to the modular reduction of $\scrB \in \Perv_G (\fr{g}, \K)$ if $[\scrA]$ appears with non-zero multiplicity in $d([\scrB])$.

	We end with a short proposition giving a tool for finding cuspidal, character, and orbital sheaves in the modular setting from the characteristic 0 setting. 

	\begin{proposition}\cite{AHJR1}
		Let $\scrG$ be a simple object in $\Perv_G (\fr{g}, \F)$.
		\begin{enumerate}
			\item If $\scrG$ occurs in the modular reduction of a cuspidal sheaf $\scrF \in \Perv_G (\fr{g}, \K)$. Then $\scrG$ is cuspidal.
			\item If $\scrG$ occurs in the modular reduction of a character sheaf $\scrF \in \Perv_G (\fr{g}, \K)$. Then $\scrG$ is a character sheaf.
			\item If $\scrG$ occurs in the modular reduction of an orbital sheaf $\scrF \in \Perv_G (\fr{g}, \K)$. Then $\scrG$ is an orbital sheaf.
		\end{enumerate}
	\end{proposition}
	\begin{proof}
		\textbf{(1): } Let $\scrG$ be a simple perverse sheaf in $\Perv_G (\fr{g}, \F)$ occurring in the modular reduction of a cuspidal sheaf $\scrF \in \Perv_G (\fr{g}, \K)$.
		In particular, let $\scrF_\O$ be a torsion-free sheaf in $\Perv_G (\fr{g}, \O)$ such that $\scrF \cong \K \otimes_\O \scrF_\O$. We have that $\scrG$ is a composition factor of $\F \otimes_\O^L \scrF_\O$.
		By performing the obvious modifications to Proposition 2.22 of \cite{AHJR1}, one obtains that if $\scrG$ is res-small. Let $\scrF \cong \scrL \boxtimes \scrB$ where $\scrL$ is a character sheaf on $Z(\fr{g})$.
		Since, $\K (-)$ gives a bijection at the level of objects, we must have that $\scrF_\O \cong \scrL_\O \boxtimes \scrB_\O$ for torsion-free sheaves $\scrL_\O \in \Perv_G (Z(\fr{g}), \O)$ and $\scrB_\O \in \Perv_G ([\fr{g}, \fr{g}], \O)$. 
		The Fourier transform commutes with extension of scalars, hence $(\T (\scrL_\O) \cong \T (\scrL) )_\O$. To check if $\scrL_\O$ is a character sheaf, one only needs to compute that $\T (\scrL_\O)$ is supported at a point. 
		Since $\supp ((\T (\scrL) )_\O) = \supp (\T (\scrL))$, we have shown that $\scrL_\O$ is a character sheaf on $Z (\fr{g})$.
		Since $\boxtimes$ commutes with $\F \otimes_\O^L (-)$, we get a $\boxtimes$-decomposition of $\F \otimes_\O^L \scrF_\O$,
		\[\F \otimes_\O^L \scrF_\O \cong (\F \otimes_\O^L \scrL_\O) \boxtimes (\F \otimes_\O \scrB_\O)\]
		Let $\scrL_\F = \F \otimes_\O^L \scrL_\O$. It remains to check that is a character sheaf, but this follows from the fact that extension of scalars commutes with pullbacks:
		\[+^* \scrL_\F \cong \F \otimes_\O^L +^* \scrL_\O \cong \F \otimes_\O^L \scrL_\O \cong \scrL_F\]  

		\textbf{(2): } By Remark 2.23 of \cite{AHJR1}, one has that induction commutes with modular reduction. The claim then follows from Theorem \ref{admissible_and_cs_are_the_same} and (1).

		\textbf{(3): } The Fourier transform is an equivalence of categories and commutes with scalars, and hence, commutes with modular reduction. The claim then follows by taking the Fourier transform of (2).
	\end{proof}

	\newpage
	\bibliography{mps_v2}{}
	\bibliographystyle{alpha}
\end{document}